\documentclass[leqno]{amsart}
\usepackage{amssymb}
\usepackage{mathrsfs}
\usepackage{amsmath, amsfonts, vmargin, enumerate}
\usepackage{graphics}
\usepackage{color}
\usepackage{verbatim}
\usepackage{amsthm}
\usepackage{latexsym, bm}
\usepackage{euscript}
\usepackage{dsfont}

\input xypic
\xyoption {all}

\makeatletter

\newtheorem{thm}{Theorem}[section]
\newtheorem{prop}[thm]{Proposition}
\newtheorem{cor}[thm]{Corollary}
\newtheorem{lem}[thm]{Lemma}
\newtheorem{defi}[thm]{Definition}
\newtheorem{remark}[thm]{Remark}
\newtheorem{example}[thm]{Example}
\newtheorem{pb}[thm]{Problem}
\newtheorem{conj}[thm]{Conjecture}

\newenvironment{rk}{\begin{remark}\rm}{\end{remark}}

\numberwithin{equation}{section}

\newcommand{\real}{{\mathbb R}}
\newcommand{\nat}{{\mathbb N}}
\newcommand{\ent}{{\mathbb Z}}
\newcommand{\com}{{\mathbb C}}
\newcommand{\un}{{\mathds {1}}}

\newcommand{\T}{{\mathbb T}}
\newcommand{\I}{{\mathbb I}}

\newcommand{\A}{{\mathcal A}}

\renewcommand{\H}{{\mathcal H}}

\newcommand{\M}{{\mathcal M}}
\newcommand{\N}{{\mathcal N}}

\newcommand{\BMO}{{\rm BMO}}

\renewcommand{\a}{\alpha}

\newcommand{\Ga}{\Gamma}
\newcommand{\D}{\Delta}

\newcommand{\e}{\varepsilon}

\newcommand{\s}{\sigma}

\newcommand{\ot}{\otimes}

\newcommand{\8}{\infty}
\newcommand{\el}{\ell}

\newcommand{\la}{\langle}
\newcommand{\ra}{\rangle}
\newcommand{\wt}{\widetilde}
\newcommand{\wh}{\widehat}
\newcommand{\n}{\noindent}

\newcommand{\les}{\lesssim}

\newcommand{\Om}{\Omega}
\newcommand{\be}{\begin{eqnarray*}}
\newcommand{\ee}{\end{eqnarray*}}
\newcommand{\beq}{\begin{equation}}
\newcommand{\eeq}{\end{equation}}
\newcommand{\beqn}{\begin{equation*}}
\newcommand{\eeqn}{\end{equation*}}

\begin{document}

\title{Characterizations of operator-valued Hardy spaces and applications to harmonic analysis on quantum tori}
\author{Runlian Xia}
\address{School of Mathematics and Statistics, Wuhan University, Wuhan 430072, China, and Laboratoire de Math{\'e}matiques, Universit{\'e} de Franche-Comt{\'e},
25030 Besan\c{c}on Cedex, France}
\email{runlian.xia@edu.univ-fcomte.fr}



\author{Xiao Xiong}
\address{Laboratoire de Math{\'e}matiques, Universit{\'e} de Franche-Comt{\'e},
25030 Besan\c{c}on Cedex, France}
\email{xiao.xiong@univ-fcomte.fr}

\author{Quanhua Xu}
\address{School of Mathematics and Statistics, Wuhan University, Wuhan 430072, China, and Laboratoire de Math{\'e}matiques, Universit{\'e} de Franche-Comt{\'e},
25030 Besan\c{c}on Cedex, France, and Institut Universitaire de France}
\email{qxu@univ-fcomte.fr}

\date{}

\maketitle

\markboth{R. Xia , X. Xiong, Q. Xu}%
{Operator-valued Hardy spaces}


\begin{abstract}
 This paper studies the operator-valued Hardy spaces introduced and studied by Tao Mei. Our principal result shows that the Poisson kernel in Mei's definition of these spaces can be replaced by any reasonable test function. As an application, we get a general characterization of Hardy spaces on quantum tori. The latter characterization plays a key role in our recent study of Triebel-Lizorkin spaces on quantum tori.
\end{abstract}

\bigskip



\section{Introduction and main results}


This paper is devoted to the study of operator-valued Hardy spaces introduced by Mei \cite{Mei2007}. Motivated by the development of noncommutative martingale inequalities (see, for instance, \cite{ju-doob, JX2003, jx-erg, LPP1991, PR2006, PX1997, Ran2002, Ran2007}) and  the Littlewood-Paley-Stein theory of quantum Markov semigroups (cf. \cite{JLX2006, JM2011, JM2010}), Mei developed a remarkable theory of operator-valued Hardy spaces on $\real^d$. These spaces are shown to be very useful for many aspects of noncommutative harmonic analysis (cf. e.g. \cite{JMP2014, JMP2015}). They are defined by the Littlewood-Paley $g$-function or Lusin area integral  function associated to the Poisson kernel. However, it is a classical result in the scalar case that the Poisson kernel does not play any special role and can be replaced by any (reasonable) test function with mild conditions. This extension is not only interesting of its own right but also crucial for applications; for instance, it plays an important role in the part of harmonic analysis related to the Littlewood-Paley decomposition as well as in the applications of harmonic analysis to PDE. 

Recently, we were led to extending this classical result to the noncommutative setting in our study of Triebel-Lizorkin spaces on quantum tori  in \cite{XXY} (see also the announcement \cite{XXYa}). This noncommutative extension is a key ingredient for the part of  \cite{XXY} on Triebel-Lizorkin spaces. To our best knowledge, all existing proofs of this result use maximal functions in a crucial way. Because of the lack of the noncommutative analogue of the pointwise maximal function, they do not extend to the operator-valued setting.  We will investigate the problem via duality as in Mei's work \cite{Mei2007}, combined with   the operator-valued Calder\'on-Zygmund theory. We show that the main arguments of \cite{Mei2007} can be adapted to general test functions in place of the Poisson kernel. This adaptation sometimes is quite straightforward, sometimes requires significantly extra efforts.  One of the major differences is the lack of harmonicity of the convolution function by a general test kernel. This harmonicity is useful for some arguments in \cite{Mei2007}; for example, it permits one to easily see the majoration of the Littlewood-Paley (radial) square function  by the Lusin  (conic) square function. In the general case, we have a variant of this result whose proof is, however, more elaborated. It should be also pointed out that both radial and conic square functions are  important for the theory: the former is simpler and readily extends to the setting of semigroups; because of the non-tangential nature of the cone used,  the latter controls other related functions and is required for the $H_1$-BMO duality and atomic decomposition. 

We would like to emphasize that the approach developed here seems new even in the scalar case.  However, it presents a drawback: due to its duality nature, it does not allow us to handle Hardy spaces $\H_p$ for $p<1$, in contrast with the classical approach by maximal functions.

The results proved and techniques developed in this paper are crucial tools in a forthcoming work \cite{Xia} of the first named author on the localization of operator-valued Hardy spaces on $\real^d$  and their applications to pseudo-differential operators. They will also play an important role in our ongoing project on operator-valued Triebel-Lizorkin spaces on $\real^d$. Like in the classical case, the latter spaces, together with the accompanying classes of Sobolev and Besov spaces, will be central objects in the study of pseudo-differential operators in the noncommutative setting. In the same spirit, one  might naturally expect that the outcome of the present investigation would be useful in the very fresh but promising direction of noncommutative PDEs.

\medskip

To state our main results, we require some preliminaries on the noncommutative $L_p$ spaces and operator-valued Hardy spaces.


\subsection{Noncommutative $L_p$-spaces.}


Let $\M$ be  a von Neumann algebra equipped with a
normal semifinite faithful trace $\tau$; for $1\le p\le\8$, let $L_p(\M)$ be the noncommutative $L_p$-space associated to $(\M,\tau)$.
The norm of  $L_p(\M)$ will be often denoted simply by $\|\,\|_p$. But if different  $L_p$-spaces  appear in a same context, we will sometimes precise the respective $L_p$-norms in order to avoid possible ambiguity. The reader is referred to \cite{PX2003} and  \cite{Xu2007} for more information on noncommutative
$L_p$-spaces. Like the classical $L_p$-spaces, noncommutative $L_p$-spaces behave well with respect to interpolation.  For instance,  for $1\le p_0<p_1\leq \infty$ and $0<\eta<1$, we have
 $$\big( L_{p_0}(\M),\, L_{p_1}(\M) \big)_{\eta}= L_{p}(\M)\;\text{ with equal norms},$$
where $\frac{1}{p}=\frac{1-\eta}{p_0}+\frac{\eta}{p_1}$ and $(\cdot\,,\,\cdot)_\eta$ denotes the complex interpolation method (see  \cite{BL1976} for interpolation theory).

We will need Hilbert space-valued noncommutative $L_p$-spaces.
Let $H$ be a Hilbert space and  $v \in H$ with $\|v\|=1$. Let  $p_v$ be the orthogonal projection onto
the one-dimensional subspace generated by $v$. Define
 $$L_p(\M; H^{r})=(p_v\ot1_{\M}) L_p(B(H)\overline\ot\M)\;\textrm{ and }\;
 L_p(\M; H^{c})= L_p(B(H)\overline\ot\M)(p_v\ot1_{\M}),$$
where the tensor product $B(H)\overline\ot\M$ is equipped with the tensor trace while $B(H)$ is equipped with the usual trace. These are the row and column noncommutative $L_p$-spaces. For $f\in L_p(\M; H^c)$,
 $$\|f\|_{L_p(\M;H^c)}=\|(f^*f)^{\frac{1}{2}}\|_{L_p(\M)}.$$
We have a similar formula for the row space by passing to adjoints: $f\in L_p(\M; H^{r})$ iff $f^*\in L_p(\M; H^c)$; and  $\|f\|_{L_p(\M;H^r)}=\|f^*\|_{L_p(\M;H^c)}$. It is clear that $ L_p(\M; H^c)$ and $ L_p(\M; H^r)$ are 1-complemented subspaces of $L_p(B(H)\overline\ot\M)$ for any $p$. Thus they also form an interpolation scale with respect to the complex interpolation method:
For $1\leq p_0, p_1\leq\8$ and $0<\eta<1$, we have
 \be\begin{split}
 \big( L_{p_0}(\M; H^c),\, L_{p_1}(\M; H^c) \big)_{\eta}&= L_{p}(\M; H^c)\;\text{ with equal norms},
\end{split}\ee
where $\frac{1}{p}=\frac{1-\eta}{p_0}+\frac{\eta}{p_1}$. The same formula holds for row spaces too.


\subsection{Operator-valued Hardy spaces.}


Throughout the remainder of the paper, unless explicitly stated otherwise,   $(\M,\tau)$ will be fixed as before  and $\N=L_\8(\real^d)\overline\ot\M$, equipped with the tensor trace. In this subsection, we introduce Mei's operator-valued Hardy spaces. Contrary to the custom, we will use letters $s, t$ to denote variables of $\real^d$ since letters $x, y$ are reserved to operators in noncommutative $L_p$-spaces. Accordingly, a generic element of the upper half-space $\mathbb{R}^{d+1}_+$ will be denoted by $(s,\e)$ with $\e>0$: $\mathbb{R}^{d+1}_+=\{(s,\e): s\in\real^d,\, \e>0\}$.

Let $\mathrm{P}$ be the Poisson kernel of $\real^d$:
 $$\mathrm{P}(s)=c_d\,\frac{1}{(|s|^2+1)^{\frac{d+1}2}}$$
with $c_d$ the usual normalizing constant and $|s|$ the Euclidean norm of $s$. Let
  $$\mathrm{P}_\e(s)=\frac1{\e^d}\, \mathrm{P}(\frac s\e)=c_d\,\frac{\e}{(|s|^2+\e^2)^{\frac{d+1}2}}\,.$$
For any function $f$ on $\mathbb{R}^d$ with values in $L_1(\M)+L_\8(\M)$,  its Poisson integral, whenever exists, will be denoted   by $\mathrm{P}_\e(f)$:
 $$\mathrm{P}_\e(f)(s)=\int_{\mathbb{R}^d}\mathrm{P}_\e(s-t)f(t)dt, \quad (s,\e)\in \mathbb{R}^{d+1}_+.$$
Note that the Poisson integral of $f$ exists if
 $$f\in L_1(\M; L^c_2(\real^d,\frac{dt}{1+|t|^{d+1}}))+L_{\infty}(\M;L^c_2(\real^d,\frac{dt}{1+|t|^{d+1}})).$$
This space is the  right space in which all functions  considered in this paper live as far as only column spaces are concerned. As it will appear frequently later, to simplify notation we will denote the Hilbert space $L_2(\real^d,\frac{dt}{1+|t|^{d+1}})$ by $\mathrm{R}_d$:
 \beq\label{Rd}
 \mathrm{R}_d=L_2(\real^d,\frac{dt}{1+|t|^{d+1}}).
 \eeq
The Lusin area square function of $f$ is defined by
 \beq\label{Lusin}
 S^c(f) (s) = \Big(\int_{\Gamma} \big|\frac{\partial}{\partial\e} \mathrm{P}_\e(f)(s+t)\big|^2\,\frac{dt\,d\e}{\e^{d-1}}\Big)^{\frac 1 2}, \quad s\in \mathbb{R}^d,
 \eeq
where $\Gamma$ is the cone $\{(t,\e)\in \real^{d+1}_+: |t|<\e\}$.
For $1\le p<\8$ define the column Hardy space $\mathcal{H}_p^c(\mathbb{R}^d, \M)$ to be
 $$\mathcal{H}_p^c(\mathbb{R}^d, \M)=\big\{f: \|f\|_{\mathcal{H}_p^c}
 =\big\|S^c(f)\|_{L_p(\N)}<\8\big\}.$$
Note that  \cite{Mei2007} uses the gradient of $\mathrm{P}_\e(f)$ instead of the sole radial derivative in the definition of $S^c$ above,  but this does not affect  $\mathcal{H}_p^c(\mathbb{R}^d, \M)$ (up to equivalent norms). On the other hand, it is proved in  \cite{Mei2007} that $\mathcal{H}_p^c(\mathbb{R}^d, \M)$ can be equally defined by the Littlewood-Paley $g$-function:
  \beq\label{LP}
 s^c(f) (s) = \Big(\int_{0}^\8 \e\, \big|\frac{\partial }{\partial\e}\mathrm{P}_\e(f)(s)\big|^2\,d\e \Big)^{\frac 1 2}, \quad s\in \mathbb{R}^d.
 \eeq
 Thus
  $$\|f\|_{\mathcal{H}_p^c}\approx \big\|s^c(f)\|_{L_p(\N)},\quad f\in \mathcal{H}_p^c(\mathbb{R}^d, \M).$$
The row Hardy space $\mathcal{H}_p^r(\mathbb{R}^d, \M)$ is the space of all $f$ such that $f^*\in\mathcal{H}_p^c(\mathbb{R}^d, \M)$, equipped with the norm
 $\|f\|_{\mathcal{H}_p^r}= \|f^*\|_{\mathcal{H}_p^c}\,.$
Finally, we define the mixture space $\mathcal{H}_p(\mathbb{R}^d, \M)$ as
 $$\mathcal{H}_p(\mathbb{R}^d, \M)=\mathcal{H}_p^c(\mathbb{R}^d, \M)+\mathcal{H}_p^r(\mathbb{R}^d, \M)
 \;\textrm{ for }\; 1\le p\le2$$
equipped with the sum norm
 $$\|f\|_{\H_p}=\inf\big\{\|f_1\|_{\H^c_p}+\|f_2\|_{\H^r_p}: f=f_1+f_2\big\},$$
and
 $$\mathcal{H}_p(\mathbb{R}^d, \M)=\mathcal{H}_p^c(\mathbb{R}^d, \M)\cap\mathcal{H}_p^r(\mathbb{R}^d, \M)
 \;\textrm{ for }\; 2< p<\8$$
equipped with the intersection norm
 $$\|f\|_{\H_p}=\max\big(\|f\|_{\H^c_p}\,, \|f\|_{\H^r_p}\big).$$
 Observe that
 $$\mathcal{H}_2^c(\mathbb{R}^d, \M)=\mathcal{H}_2^r(\mathbb{R}^d, \M)
 =L_2(\N)\;\textrm{ with equivalent norms.}$$
It is proved in  \cite{Mei2007} that for $1<p<\8$
 $$ \mathcal{H}_p(\mathbb{R}^d, \M)=L_p(\N)\;\textrm{ with equivalent norms.}$$

The operator-valued BMO spaces are also studied in \cite{Mei2007}. Let $Q$ be a cube in $\real^d$ (with sides parallel to the axes) and $|Q|$ its volume. For a function $f$ with values in $\M$, $f_Q$ denotes its mean over $Q$:
 $$f_Q = \frac{1}{|Q|} \int_Q f(t)dt.$$
The column BMO norm of $f$ is defined to be
 \begin{equation}
 \|f\|_{\mathrm{BMO}^c}=\sup_{Q\subset\real^d}\Big\|\frac{1}{|Q|}\int_Q\big|f(t)-f_Q\big|^2dt\Big\|_{\M}^{\frac 12}.
 \end{equation}
Then
 $$\mathrm{BMO}^c(\real^d,\M)=\big\{f: f \in
 L_{\infty}\big(\M; \mathrm{R}_d^c\big),\; \|f\|_{\mathrm{BMO}^c}<\infty\big\}.$$
Similarly, we define the row space $\mathrm{BMO}^r(\real^d,\M)$ as the space of $f$ such that $f^*\in \mathrm{BMO}^c(\real^d,\M)$,  and  $\mathrm{BMO}(\real^d,\M)=\mathrm{BMO}^c(\real^d,\M)\cap \mathrm{BMO}^r(\real^d,\M)$  with the intersection norm.

\begin{rk}
It is easy to see that in the above definition of $\BMO^c$, cubes $Q$ can be replaced by balls $B$ without changing $\BMO^c$ (up to equivalent norms). In the sequel, we will use cubes or balls according to problems in consideration.
\end{rk}

One of the main results of  \cite{Mei2007} asserts that the dual of  $\mathcal{H}_1^c(\mathbb{R}^d, \M)$ can be naturally identified with $\mathrm{BMO}^c(\real^d,\M)$. This is the operator-valued analogue of the celebrated Fefferman $H_1$-BMO duality theorem. The following interpolation result is also taken from \cite{Mei2007}.

\begin{lem}\label{Hardyinterpolation}
Let $1<p<\infty$. Then
 $$\big(\mathrm{BMO}^c(\real^d,\M),\; \mathcal{H}_1^c(\real^d,\M)\big)_{\frac{1}{p}}=\mathcal{H}^c_p(\real^d,\M)\;\text{ with equivalent norms}.$$
Similar statements hold for the row and mixture spaces too.
\end{lem}


\subsection{ Main results.}


We now announce the main results of the paper. They assert that the Poisson kernel in the definition of Hardy spaces introduced in the previous subsection can be replaced by more general test functions. Most of the time, our test functions belong to the  Schwartz class $\mathcal{S}$ of $\real^d$. Given a function $\Phi$ on $\real^d$, we set $\Phi _\e(s) = \e^{-d} \Phi(\frac s  \e)$ for $\e>0$. Now let $\Phi\in\mathcal{S}$ be of vanishing mean.  We will assume that  $\Phi$ is nondegenerate in the following sense:
 \beq\label{schwartz}
  \forall\,\xi\in\real^d\setminus\{0\}\;\exists\, \e>0 \;\text { such that }\; \wh\Phi(\e\xi)\neq0.
  \eeq
Then  there exists  $\Psi \in \mathcal{S}$ of vanishing mean  such that
  \beq\label{reproduce}
  \int_0^\8 \wh\Phi(\e\xi)\, \overline{\wh\Psi (\e\xi)}\, \frac{d\e}{\e} = 1,\quad \forall\xi\in\real^d\setminus\{0\}.
   \eeq
This is a well-known elementary fact (cf. e.g., \cite[p. 186]{Stein1993}). Indeed, choose a nonnegative infinitely differentiable function $\eta$, compactly supported and vanishing near the origin, such that $|\widehat{\Phi}|^2\eta$ does not vanish identically on any ray emanating from the origin. Let
 $$h(\xi)=\int_0^\8|\wh \Phi(\e\xi)|^2\eta(\e\xi)\,\frac{d\e}\e \,.$$
Then the function $\Psi$ determined by
 $$\wh\Psi(\xi)=\frac{\wh\Phi(\xi)\eta(\xi)}{h(\xi)}$$
does the job.

Let $f$ be a function  in $L_1(\M; \mathrm{R}_d^c)+L_{\infty}(\M; \mathrm{R}_d^c)$ (recalling that the Hilbert space $\mathrm{R}_d$ is defined by \eqref{Rd}). Then the convolution $\Phi_\e *f $ is well-defined and takes values in $L_1(\M)+L_\8(\M)$. The radial and conic square functions of $f$ associated to $\Phi$ are defined as $s^c(f)$ in \eqref{LP} and $S^c(f)$ in \eqref{Lusin} with $\Phi$ in place of the radial partial derivative  of  the Poisson kernel $ \mathrm{P}$:
 \be\begin{split}
 s_{\Phi}^c(f) (s) &= \Big(\int_0^\8 | \Phi_\e *f (s)|^2\frac{d\e}{\e}\Big)^{\frac 1 2}\,,\\
 S_{\Phi}^c(f) (s) &=\Big (\int_{\Gamma} | \Phi_\e *f (s+t)|^2\frac{dtd\e}{\e^{d+1}}\Big)^{\frac 1 2} \,, \quad s\in\real^d.
 \end{split}\ee
Unless explicitly stated otherwise, $f$ will be always assumed to belong to $L_1(\M; \mathrm{R}_d^c)+L_{\infty}(\M; \mathrm{R}_d^c)$ in the sequel, whenever $\Phi_\e *f $ is considered; sometimes, we need to impose more regularity to $f$ in order to legitimate  the relevant calculations.

Throughout the paper, we will frequently use the notation $A\les B$, which is an inequality up to a constant: $A\le c\, B$ for some constant $c>0$. The relevant constants in all such inequalities may depend on the dimension $d$, the test function $\Phi$ or $p$, etc. but never on the functions $f$ in consideration. An equivalence $A\approx B$ will mean $A\les B$ and $B\les A$.

The following is one of the main results of this paper.

\begin{thm}\label{equivalence Hp}
Let $1\le p<\8$ and $f\in L_{1}(\M; \mathrm{R}_d^c)+L_{\infty}(\M; \mathrm{R}_d^c)$. Then $f\in\H_p^c(\real^d,\M)$ iff  $s_{\Phi}^c(f)\in L_{p}(\N)$ iff $S_{\Phi}^c(f)\in L_{p}(\N)$. If this is the case, then
 $$\|s_{\Phi}^c(f)\|_{L_{p}(\N)} \approx \|S_{\Phi}^c(f)\|_{L_{p}(\N)} \approx \|f\|_{\mathcal{H}^c_p}$$
with relevant constants depending only on $p, d$ and $\Phi$.
\end{thm}

The above square functions $s_{\Phi}^c$ and $S_{\Phi}^c$ can be discretized as follows:
 \be\begin{split}
 s_{\Phi}^{c,D}(f)(s)  &= \Big(\sum_{j=-\8}^{\8} |\Phi_{2^j}*f (s)|^2\Big)^{\frac 1 2}\,,\\
 S_{\Phi}^{c,D}(f)(s)  &= \Big(\sum_{j=-\8}^{\8} 2^{-dj}\int_{B(s, 2^j)} |\Phi_{2^j}*f (t)|^2 dt\Big)^{\frac 1 2}\,.
  \end{split}\ee
Here $B(s, r)$ denotes the ball of $\real^d$ with center $s$ and radius $r$. To prove that these discrete square functions also describe our Hardy spaces, we need to impose the following  condition to the previous Schwartz function $\Phi$ of vanishing mean, which is stronger than \eqref{schwartz}:
 \beq\label{schwartz D}
  \forall\,\xi\in\real^d\setminus\{0\}\;\exists\, 0<2a\le b<\8\;\text { such that }\; \wh\Phi(\e\xi)\neq0,\;\forall\; \e\in (a,\,b].
  \eeq
 Then adapting the proof of \cite[Lemma~V.6]{ST1989} , we can find another Schwartz function $\Psi$ such that
 \beq\label{reproduceD}
   \sum_{j=-\8}^\8 \wh\Phi(2^j\xi)\, \overline{\wh\Psi (2^j\xi)} = 1,\quad \forall\xi\in\real^d\setminus\{0\}.
 \eeq
 The following discrete version of Theorem \ref{equivalence Hp} plays a crucial role in the study \cite{XXY} of  Triebel-Lizorkin spaces on quantum tori.

\begin{thm}\label{equivalence HpD}
Let $1\le p<\8$ and $f\in L_{1}(\M; \mathrm{R}_d^c)+L_{\infty}(\M; \mathrm{R}_d^c)$. Then $f\in\H_p^c(\real^d,\M)$ iff  $s_{\Phi}^{c, D}(f)\in L_{p}(\N)$ iff $S_{\Phi}^{c, D}(f)\in L_{p}(\N)$. Moreover,
 $$\|s_{\Phi}^{c, D}(f)\|_{L_{p}(\N)} \approx \|S_{\Phi}^{c, D}(f)\|_{L_{p}(\N)} \approx \|f\|_{\mathcal{H}^c_p}$$
with relevant constants depending only on $p, d$ and $\Phi$.
\end{thm}

The requirement that $\Phi\in\mathcal{S}$ can be considerably relaxed in the preceding two theorems. Here, we consider only one example: $\Phi=I^\a(\mathrm{P})$ with $\a>0$, where $I^\a$ is  the Riesz potential of order $\a$.  Recall that $I^\a=(-(2\pi)^{-2}\D)^{\frac \a 2}$ is a Fourier multiplier on $\real^d$ with symbol $I_\a$ defined by $I_\a(\xi)=|\xi|^\a$. Thus $\wh{I_\a(\mathrm{P}})(\xi)=|\xi|^\a \wh{\mathrm{P}}(\xi)=|\xi|^\a e^{-2\pi|\xi|}$. Then the preceding two theorems continue to hold for this choice of $\Phi$. We only state the following radial version which is used in the  proof of the  Poisson semigroup characterization of noncommutative Triebel-Lizorkin spaces in \cite{XXY}.

\begin{thm}\label{charact-Riesz of Poisson}
Let $\Phi=I^\a(\mathrm{P})$ with $\a>0$. Then for any $1\leq p<\8$, we have
\beq\label{Riesz of Poisson}
\|f\|_{\mathcal{H}^{c}_p}\approx \Big\|\Big(\int_0^\8 \big|\Phi_{\e}*f\big|^2\frac{d\e}{\e}\Big)^{\frac12}\Big\|_{L_p(\N)}
\eeq
with relevant constants depending only on $p, d$ and $\a$. Consequently,  for any integer $k\ge1$
\beq\label{D of Poisson}
\|f\|_{\mathcal{H}^{c}_p}\approx\Big\|\Big(\int_0^\8\e^{2k} \big|\frac{\partial^k}{\partial \e^k} \mathrm{P}_\e(f)\Big|^2\frac{d\e}{\e}\Big)^{\frac{1}{2}}\Big\|_{L_p(\N)}\,.
\eeq
\end{thm}

\begin{rk}
Note that letting $k=1$ in \eqref{D of Poisson}, we return back to the original definition of Hardy spaces in Mei \cite{Mei2007}.  We also would  like to point out that the above theorem seems new even in the scaler case; compare it with the characterization of Triebel-Lizorkin spaces $F_{p,q}^\a(\real^d)$ ($F_{p,2}^0(\real^d)=\mathcal{H}_p(\real^d)$) in \cite[Theorem~2.6.4]{HT1992}.
\end{rk}

\begin{rk}
As mentioned before, the preceding three theorems play an important role in the proof of the general characterization of noncommutative Triebel-Lizorkin spaces in \cite{XXY}.  Conversely, the latter can be used to  characterize $\H_p^c(\real^d,\M)$ by test functions much more general than  $\Phi$ in the preceding theorems. This will be pursued elsewhere.
\end{rk}

\medskip

The paper is organized as follows. The next section contains some elementary results on Calder\'on-Zygmund operators in the noncommutative setting.  In section~\ref{Carleson measures}, we establish the link between Carleson measures and BMO spaces. Sections~\ref{Proof1}, \ref{Proof2} and \ref{A special Phi} are devoted, respectively, to the proofs of Theorems \ref{equivalence Hp}, \ref{equivalence HpD} and \ref{charact-Riesz of Poisson}. Sections~\ref{Applications to tori} and \ref{Applications to quantum tori} present applications to the usual and quantum tori.


\section{Calder\'on-Zygmund operators and square functions}


Let $K$ be an $L_1(\M)+L_\8(\M)$-valued distribution on $\real^d$. We will assume that $K$ coincides on $\real^d\setminus\{0\}$ with a locally integrable $L_1(\M)+L_\8(\M)$-valued function. Then the convolution $K*f$ is defined for sufficiently nice function $f$ with values in $L_1(\M)\cap L_\8(\M)$. This is the (left) singular integral operator $K^c$ associated to $K$:
 $$K^c(f)(s)=K*f(s)=\int_{\real^d}K(s-t)f(t)dt,\quad s\in\real^d.$$
For instance, $K^c(f)$ is well defined for any $f\in \mathcal S\ot(L_1(\M)\cap L_\8(\M))$ (recalling that $\mathcal S$ is the Schwartz class on $\real^d$). Note that $K^c$ is right $\M$-modular.

Similarly, we  define the right singular integral operator  $K^r$:
 $$K^r(f)(s)=f*K(s)=\int_{\real^d}f(t)K(s-t)dt.$$

We will frequently use the following Cauchy-Schwarz type inequality for the operator square function. Let $(\Om,\mu)$ be a measure space. Then
 \beq\label{CS}
 \Big|\int_{\Om}\phi fd\mu\Big|^2\le \int_{\Om}|\phi|^2d\mu \int_{\Om} |f|^2d\mu,
 \eeq
where $\phi:\Om\to\com$ and $f:\Om\to L_1(\M)+L_\8(\M)$ are  functions such that all members of the above inequality make sense.  We will also require the operator-valued version of the Plancherel formula. For sufficiently nice functions $f, g:\real^d\to L_1(\M)+L_\8(\M)$, for instance, for $f, g\in L_2(\real^d)\ot L_2(\M)$,  we have
  \beq\label{Plancherel}
  \int_{\real^d}g^*(s)f(s)ds=\int_{\real^d}(\wh g(\xi))^*\wh f(\xi) d\xi\;\text{ as measurable operators}.
  \eeq

The following result must be known to experts. It is closely related to similar results of  \cite{HLMP2014,JMP2014, MP2009, Parcet2009}. We include a proof by standard arguments for completeness. Let  $\BMO^c_0(\real^d, \M)$ denote the subspace of  $\BMO^c(\real^d, \M)$ consisting of compactly supported functions.

\begin{lem}\label{CZ}
 Assume that
  \begin{enumerate}[\rm a)]
 \item the Fourier transform of $K$ is bounded: $\displaystyle\sup_{\xi\in\real^d}\|\wh K(\xi)\|_\M<\8$;
 \item $K$ has the Lipschitz regularity: there exists a positive constant $C$ such that
 $$\|K(s-t)-K(s)\|_\M\le \frac{C\,|t|}{|s-t|^{d+1}}\,,\quad\forall |s|>2|t|.$$
 \end{enumerate}
Then $K^c$ is bounded on $\H_p^c(\real^d,\M)$ for $1\le p<\8$ and from $\BMO^c_0(\real^d, \M)$ to $\BMO^c(\real^d, \M)$.

A similar statement also holds for $K^r$ and the corresponding row spaces.
 \end{lem}

 \begin{proof}
 First suppose that $K^c$ maps constant functions to zero. This amounts to requiring that $K^c(\un_{\real^d})=0$. Let $f\in\BMO^c_0(\real^d,\M)$ and $Q$ be a cube with center  $c$. Let $\wt Q=2Q$ be the cube with center $c$ and twice the side length of $Q$.  Decompose $f$ as $f=f_1+f_2+f_{\wt Q}$ with $f_1=(f-f_{\wt Q})\un_{\wt Q}$. Then $K^c(f)=K^c(f_1)+K^c(f_2)$.  Letting
 $$\a=\int_{\real^d\setminus\wt Q}K(c-t)\big(f(t)-f_{\wt Q}\big)dt,$$
we have
 $$K^c(f)(s)-\a=K^c(f_1)(s)+\int_{\real^d}(K(s-t)-K(c-t))f_2(t)dt.$$
 Thus by \eqref{CS},
 $$\frac1{|Q|}\int_Q|K^c(f)(s)-\a|^2ds
 \le 2(A+B),$$
 where
 \be\begin{split}
 A &=\frac1{|Q|}\int_Q|K^c(f_1)(s)|^2ds, \\
 B&=\frac1{|Q|}\int_Q \Big|\int_{\real^d}(K(s-t)-K(c-t))f_2(t)dt\Big|^2ds.
  \end{split}\ee
The first term $A$ is easy to estimate. Indeed, by \eqref{Plancherel},
 \be\begin{split}
  |Q|A
  &\le\int_{\real^d}|K^c(f_1)(s)|^2ds
 =\int_{\real^d}|\wh K(\xi)\wh f_1(\xi)|^2d\xi\\
 &=\int_{\real^d}\wh f_1(\xi)^*\wh K(\xi)^*\wh K(\xi)\wh f_1(\xi) d\xi
 \le \int_{\real^d}\|\wh K(\xi)\|_\M^2|\wh f_1(\xi)|^2 d\xi\\
 &\les \int_{\real^d}|f_1(s)|^2ds
 =\int_{\wt Q}|f(s)-f_{\wt Q}|^2ds \les|\wt Q|\, \|f\|^2_{\BMO^c}\,,
 \end{split}\ee
so
 $\|A\|_{\M}\les \|f\|^2_{\BMO^c}$.

 To estimate $B$, using \eqref{CS} and the Lipschitz regularity ($l$ denoting the side length of $Q$), for any $s\in Q$,  we have
  \be\begin{split}
  &\Big|\int_{\real^d}(K(s-t)-K(c-t))f_2(t)dt\Big|^2\\
 &\le \int_{\real^d\setminus\wt Q}\|K(s-t)-K(c-t))\|_\M dt\\
 &~~~\cdot
 \int_{\real^d\setminus\wt Q}\|K(s-t)-K(c-t))\|^{-1}_\M\,\big|(K(s-t)-K(c-t))f_2(t)\big|^2dt\\
 &\les \int_{\real^d\setminus\wt Q}\|K(s-t)-K(c-t))\|_\M\,|f_2(t)|^2dt\\
 &\les l \int_{\real^d\setminus\wt Q}\frac1{|t-c|^{d+1}}\,|f_2(t)|^2d\\
 &\les \sum_{k\ge0} 2^{-k}\,\frac1{|2^{k+1}\wt Q|}\int_{2^{k+1}\wt Q\setminus2^k\wt Q}|f(t)-f_{\wt Q}|^2dt\\
 &\les \sum_{k\ge0} 2^{-k}\,\frac1{|2^{k+1}\wt Q|}\int_{2^{k+1}\wt Q}|f(t)-f_{2^{k+1}\wt Q}|^2dt+
 \sum_{k\ge0} 2^{-k}|f_{2^{k+1}\wt Q}-f_{\wt Q}|^2\\
 &\les \sum_{k\ge0} 2^{-k}\, \|f\|^2_{\BMO^c}+
 \sum_{k\ge0} 2^{-k}|f_{2^{k+1}\wt Q}-f_{\wt Q}|^2
  \end{split}\ee
 However, by \eqref{CS} once more,
  \be\begin{split}
  \big|f_{2^{k+1}\wt Q}-f_{\wt Q}\big|^2
 &\le (k+1)\sum_{j=0}^k\big|f_{2^{j+1}\wt Q} -f_{2^j\wt Q}\big|^2\\
 &\le \frac{k+1}{2^d}\sum_{j=0}^k\frac1{|2^{j+1}\wt Q|} \int_{2^{j+1}\wt Q}\big|f(t)-f_{2^{j+1}\wt Q}\big|^2dt\\
 &\le \frac{(k+1)^2}{2^d}\,  \|f\|^2_{\BMO^c}.
 \end{split}\ee
Combining the previous inequalities, we then deduce
 $\|B\|_{\M}\les \|f\|^2_{\BMO^c}$.
Therefore, $K^c$ is bounded on $\BMO^c(\real^d,\M)$.

Now it is easy to get rid of the additional requirement that $K^c(\un_{\real^d})=0$. Indeed, combining the preceding argument and the proof of \cite[Proposition~II.5.15]{gar-rubio},  we can show that $K^c(\un_{\real^d})$ can be naturally defined as a function in $\BMO^c(\real^d,\M)$. Then for $f$ and $Q$ as  above, we have $K^c(f)=K^c(f_1)+K^c(f_2)+K^c(\un_{\real^d})f_{\wt Q}$, so
 \be\begin{split}
 \|K^c(f)\|_{\BMO^c}
 &\le \|K^c(f_1)\|_{\BMO^c}+\|K^c(f_2)\|_{\BMO^c}+\|K^c(\un_{\real^d})\|_{\BMO^c}\,\|f_{\wt Q}\|_{\M}\\
 &\les  \|f\|_{\BMO^c}+\|f_{\wt Q}\|_{\M}\les \|f\|_{\BMO^c}\,.
   \end{split}\ee
 Thus we have proved the $\BMO^c$-boundedness of $K^c$ in the general case.

By duality, the boundedness of $K^c$ on $\H_1^c(\real^d,\M)$ is equivalent to that of its adjoint map $(K^c)'$ on $\BMO^c_0(\real^d,\M)$. However, it is easy to see that $(K^c)'$  is also a singular integral operator:
 $$(K^c)'(g)=\int_{\real^d}\wt K(s-t)g(t)dt,$$
where $\wt K(s)=K(-s)^*$. Clearly, $\wt K$ satisfies the same assumption as $K$, thus $(K^c)'$ is bounded on $\BMO^c_0(\real^d,\M)$, so is $K^c$ on $\H_1^c(\real^d,\M)$.

It remains to interpolate the previous two cases by means of Lemma~\ref{Hardyinterpolation}. We need, however, to note that Lemma~\ref{Hardyinterpolation} still holds with $\BMO^c_0(\real^d,\M)$ in place of $\BMO^c(\real^d,\M)$. Thus $K^c$ is bounded on  $\H_p^c(\real^d,\M)$ for any $1<p<\8$, so the assertion is proved.
  \end{proof}

A special case of Lemma \ref{CZ} concerns Hilbert-valued kernels. Let $H$ be a Hilbert space, and let $\mathsf k:\real^d\to H$ be a $H$-valued kernel. We view the vectors of $H$ as column matrices in $B(H)$ in a fixed orthonormal basis. Put $K(s)=\mathsf k(s)\ot 1_{\M}\in B(H)\overline\ot\M$. We consider the restriction of the associated singular integral operator $K^c$ to $L_2(\N)$, still denoted by the same symbol:
 $$K^c(f)(s)=K*f(s)=\int_{\real^d}K(s-t)f(t)dt$$
for nice functions $f:\real^d\to L_1(\M)+L_\8(\M)$. So $K^c$ maps functions with values in $L_1(\M)+L_\8(\M)$ to those with values in the column subspace of  $L_1(B(H)\overline\ot\M)+L_\8(B(H)\overline\ot\M)$. Consequently,
 $$\|K^c(f)\|_{L_p(B(H)\overline\ot \N)}=\|K^c(f)\|_{L_p(\N; H^c)}\,.$$
Since $\mathsf k(s)\ot 1_{\M}$ commutes with $\M$, $K^c(f)=K^r(f)$ for $f\in L_2(\N)$. Let us denote  this common operator by $\mathsf k^c$. Here the superscript $c$ refers to the previous convention that the vectors of $H$ are identified with  column matrices in $B(H)$.  Thus Lemma \ref{CZ} implies the following

\begin{cor}\label{CZH}
 Assume that
  \begin{enumerate}[\rm a)]
 \item $\displaystyle\sup_{\xi\in\real^d}\|\wh{\mathsf k}(\xi)\|_H<\8$;
 \item $\displaystyle \|\mathsf k(s-t)-\mathsf k(s)\|_H\les \frac{|t|}{|s-t|^{d+1}}\,,\quad\forall |s|>2|t|>0.$
 \end{enumerate}
 Then the operator $\mathsf k^c$ is bounded
 \begin{enumerate}[\rm i)]
 \item from $\BMO^c_0(\real^d, \M)$ to $\BMO^c(\real^d, B(H)\overline\ot\M)$, and from $\H_p^c(\real^d,\M)$ to $\H_p^c(\real^d,B(H)\overline\ot\M)$ for $1\le p<\8$;
 \item from $\BMO^r_0(\real^d, \M)$ to $\BMO^r(\real^d, B(H)\overline\ot\M)$, and from $\H_p^r(\real^d,\M)$ to $\H_p^r(\real^d,B(H)\overline\ot\M)$ for $2\le p<\8$;
 \item from $L_p(\N)$  to $L_p(\N;H^c)$ for $2\le p<\8$.
 \end{enumerate}
 \end{cor}

 \begin{proof}
  Part i) follows immediately from Lemma \ref{CZ}. Since $K^c(f)=K^r(f)$ on the subspace $L_p(\N)\subset L_p(B(H)\overline\ot \N)$, the same lemma implies the $\BMO^r$ part of ii). Interpolating this with the obvious $\H_2^r$-boundedness of $\mathsf k^c$ via Lemma~\ref{Hardyinterpolation}, we deduce that $\mathsf k^c$ is bounded from $\H_p^r(\real^d,\M)$ to $\H_p^r(\real^d,B(H)\overline\ot\M)$ for $2\le p<\8$. Combining i) and ii), we see that $\mathsf k^c$ is bounded from $\H_p(\real^d,\M)$ to $\H_p(\real^d,B(H)\overline\ot\M)$ for $2\le p<\8$. Using the equality $\H_p=L_p$ for $2\le p<\8$, we get iii).
 \end{proof}

The column subspace of $\BMO^c(\real^d, B(H)\overline\ot\M)$  (resp. $\H_p^c(\real^d,B(H)\overline\ot\M)$)  will be denoted by $\BMO^c(\real^d, H^c\overline\ot\M)$ (resp. $\H_p^c(\real^d, H^c\overline\ot\M)$). $\BMO^c(\real^d, H^c\overline\ot\M)$ and  $\H_p^c(\real^d, H^c\overline\ot\M)$ are clearly complemented in  $\BMO^c(\real^d, B(H)\overline\ot\M)$  and $\H_p^c(\real^d,B(H)\overline\ot\M)$, respectively. Similarly, we introduce the corresponding subspaces $\BMO^r(\real^d, H^c\overline\ot\M)$ and  $\H_p^r(\real^d, H^c\overline\ot\M)$.

Considered as an operator with values in these subspaces, $\mathsf k^c$ admits as adjoint  the following operator:
 $$(\mathsf k^c)'(F)(s)=\int_{\real^d}(\,\wt{\mathsf k}(s)\ot 1_{\M}) F(t)dt,$$
where $\wt{\mathsf k}(s)=\mathsf k(-s)^*$ (so it is a row matrix). The preceding corollary can be reformulated as

\begin{cor}\label{CZHdual}
 Under the same assumption, the operator $(\mathsf k^c)'$ is bounded
 \begin{enumerate}[\rm i)]
 \item from $\BMO^c_0(\real^d, H^c\overline\ot\M)$ to $\BMO^c(\real^d, \M)$,  and from $\H_p^c(\real^d,H^c\overline\ot\M)$ to $\H_p^c(\real^d,\M)$ for $1\le p<\8$;
 \item from $\H_p^r(\real^d,H^c\overline\ot\M)$ to $\H_p^r(\real^d,\M)$ for $1\le p\le2$;
 \item from $L_p(\N;H^c)$ to $L_p(\N)$ for $1<p\le2$.
 \end{enumerate}
 \end{cor}

 \begin{rk}\label{CZHbis}
Since $\H^c_p(\real^d,\M)\subset L_p(\N)$ for $1\le p\le 2$, Lemma~\ref{CZ}  implies
 $$\|K^c(f)\|_{L_p(\N)}\les \|f\|_{\H^c_p}\,,\quad\forall f\in \H^c_p(\real^d,\M).$$
In the same way, Corollary~\ref{CZH} yields
 $$\|\mathsf k^c(f)\|_{L_p(\N; H^c)}\les \|f\|_{\H^c_p}\,,\quad\forall f\in \H^c_p(\real^d,\M).$$
\end{rk}

We now apply the above theory to the square function operators $s_{\Phi}^c$ and $S_{\Phi}^c$. It is well known that these operators  can be expressed as  Calder\'on-Zygmund operators with Hilbert-valued kernels. Let us explain this  for $s_{\Phi}^c$. Let $H=L_2((0,\8),\frac {d\e}{\e})$ and define the kernel $\mathsf k:\real^d\to H $ by  $\mathsf k (s)
= \Phi_{\cdot}(s)$ ($\Phi_{\cdot}(s)$ being the function $\e\mapsto \Phi_\e(s)$). Then one easily checks that
 $$\sup_{\xi\in\real^d}\|\wh {\mathsf k}(\xi)\|_H<\8\;\textrm{ and }\;
 \|\nabla {\mathsf k}(s)\|_H\leq \frac{c}{|s|^{d+1}}\,,\quad\forall s\in\real^d\setminus\{0\}.$$
Thus ${\mathsf k}$ satisfies the assumption of Corollary~\ref{CZH}. It is clear that
  $$s_{\Phi}^c(f) (s) =\|{\mathsf k}^c(f)(s)\|_H.$$
The treatment of $S_{\Phi}^c$ is similar; this time, the Hilbert space $H$ is  $L_2(\Ga,\frac{dtd\e}{\e^{d+1}})$. Moreover, using the Plancherel formula and \eqref{reproduce}, one easily sees that
 $$
\|s_{\Phi}^c(f)\|_{ L_2(\N)} \approx \|f\|_{ L_2(\N)}\approx \|S_{\Phi}^c(f)\|_{ L_2(\N)}\,, \quad \forall f\in L_2(\N),
 $$
where the equivalence constants depend only on $\Phi$. Thus by Remark~\ref{CZHbis}, we get

\begin{lem}\label{SboundedH1toL1}
 Let $1\le p\le2$. Then
$$\max\big(\|s_{\Phi}^c(f)\|_{L_p(\N)}, \, \|S_{\Phi}^c(f)\|_{L_p(\N)}\big)
\les\|f\|_{\mathcal{H}^c_p}\,,\quad\forall f\in \H_p^c(\real^d,\M).$$
\end{lem}

Note that in the scalar case (i.e., $\M=\com$), Corollary~\ref{CZH} implies that the above lemma holds for $2<p<\8$ too. Then one easily deduces the reverse inequality by duality for $1<p<\8$. Indeed, for $f\in \H_p(\real^d)$ (with $\M=\com$) choose $g\in \H_q(\real^d)$ such that
 $$\int_{\real^d}f(s)\bar g(s)ds\approx \|f\|_{\H_{p}} \;\textrm{ and }\; \|g\|_{\H_{q}}\le1,$$
where $q$ is the conjugate index of $p$. Then by \eqref{reproduce} and the H\"older inequality
 \be\begin{split}
 \int_{\real^d}f(s)\overline{g(s)}ds
 &=\int_{\real^{d+1}_+}\Phi_\e*f(s)\,\overline{\Psi_\e* g(s)}ds\,\frac{d\e}{\e}\\
& \le \|s_{\Phi}(f)\|_{p}\, \|s_{\Psi}(g)\|_{q}
 \les \|s_{\Phi}(f)\|_{p}\, \|g\|_{\H_{q}}\les \|s_{\Phi}(f)\|_{p}\, .
  \end{split}\ee
This simple argument does not, unfortunately, apply to the case $p=1$ which is much subtler. However, in the operator-valued setting, the case $1<p<2$ seems hard too.


\section{Carleson measures}
\label{Carleson measures}


A duality argument on $\H_1$ involves unavoidably $\BMO$.  Thus we need a  square function characterization of $\BMO$ by general test functions. This is done by means of Carleson measures. In this section, $\Phi$ is a Schwartz function of vanishing mean and satisfies \eqref{schwartz}.

\begin{lem}\label{CarlsonlessBMO}
 Let $f\in \BMO^c(\real^d,\M)$ and
 $$d\mu(f) =  |\Phi_\e*f (s)|^2\frac{ds d\e}{\e} .$$
Then $d\mu$ is an $\M$-valued Carleson measure on $\real^{d+1}_+$ in the  sense  that
 $$
 \|d\mu(f))\|_{\rm C}\;{\mathop=^{\rm def}}\;\sup_B \Big\|\frac{1}{|B|}\int_{T(B)}|\Phi_\e*f (s)|^2\frac{ds d\e}{\e}\Big\|_{\M}<\8,
 $$
where the supremum runs over all balls $B\subset \real^d$,  and where $T(B)=B\times (0,\,r]$ with $r$ the radius of $B$. Moreover,
 $$\|d\mu(f)\|_{\rm C}\les \|f\|^2_{\BMO^c}.$$
\end{lem}

\begin{proof}
Given a ball $B$, we decompose $f=f_1+f_2+f_3$, where $f_1 =(f-f_{2B})\un_{2B}$ and $f_2=(f-f_{2B})\un_{\real\setminus 2B}$.  Since $\Phi$ is of vanishing mean, we have $\Phi_\e*f = \Phi_\e *f_1+ \Phi_\e*f_2$. Let  $d\mu=d\mu(f)$, $d\mu_1=d\mu(f_1)$ and $d\mu_1=d\mu(f_2)$.  Then by \eqref{CS},
 $$d\mu\le 2 (d\mu_1+ d\mu_2).$$
We first deal with $d\mu_1$. By  \eqref{Plancherel} , we have
 \be\begin{split}
\int_B s^c_{\Phi}(f_1)(s)^2ds
& \leq  \int_0^\8\int_{\real^d} |\Phi_\e*f_1 (s)|^2\frac{ds d\e}{\e} \\
& = \int_0^\8\int_{\real^d} |\widehat{\Phi}(\e\xi)|^2|\widehat{f_1}(\xi)|^2 d\xi \,\frac{d\e}{\e}\\
& \les \int_{\real^d} |f_1(s)|^2ds \les \int_{2B} |f-f_{2B}|^2 ds\les |B|\cdot \|f\|_{\BMO^c}^2\,.
 \end{split}\ee
However,
 $$
\int_{T(B)} |\Phi_\e *f_1(s)|^2 \frac{ds d\e}{\e}\le \int_{B} \int_0^\8 |\Phi_\e*f_1(s)|^2 \frac{d\e}{\e}\,ds
=\int_B s^c_{\Phi}(f_1)(s)^2ds.
 $$
It then follows that
 $\|d\mu_1\|_{\rm C} \les \|f\|_{\BMO^c}^2\,.$

On the other hand, let $s_0$ be the center of $B$ and $r$ its radius. Then for $(s,\e)\in T(B)$, by \eqref{CS}
 $$
 | \Phi_\e*f_2(s)|^2 \les \int_{\real^d \setminus 2B} \frac{\e\,|f(t)- f_{2B} |^2}{(\e+|t-s_0|)^{d+1}}dt
 \les \e \int_{\real^d \setminus 2B} \frac{|f(t)- f_{2B} |^2}{|t-s_0|^{d+1}}dt.
 $$
The last integral can be estimated by standard arguments as follows (see also the proof of Lemma~\ref{CZ}):
 \be\begin{split}
  \int_{\real^d \setminus 2B} \frac{|f(t)- f_{2B} |^2}{|t-s_0|^{d+1}}dt
 &=\sum_{k\ge1}\int_{2^{k+1}B\setminus 2^kB}\frac{|f(t)- f_{2B} |^2}{|t-s_0|^{d+1}}dt\\
 &\les\frac1r \sum_{k\ge1}2^{-k}\frac1{\big|2^{k+1}B\big|}\int_{2^{k+1}B}|f(t)- f_{2B} |^2dt
 \les \frac1r\,\|f\|_{\BMO^c}^2.
  \end{split}\ee
 Thus
  $$\frac1{|B|}\int_{T(B)}|\Phi_\e*f_2(s)| ^2\frac{d\e ds}{\e}\les \|f\|_{\BMO^c}^2.$$
Namely, $\|d\mu_2\|_{\rm C} \les \|f\|_{\BMO^c}^2$.
\end{proof}

The above argument is modeled on the classical pattern; see, for instance, the proof of \cite[Theorem~IV.4.3]{Stein1993}. In fact, our operator-valued case can be easily deduced from the classical one. By definition, we see that
 $$
 \|d\mu(f)\|_{\rm C}=\sup_{v\in H,\,\|v\|_H=1}\sup_B  \frac{1}{|B|}\int_{T(B)}\big\|\Phi_\e*f_v (s)\big\|_H^2\frac{ds\,d\e}{\e},
  $$
 where $H$ is the Hilbert space on which $\M$ acts and $f_v(s)=f(s)v$. On the other hand, we also have
 $$\|f\|_{\BMO^c}=\sup_{v\in H,\,\|v\|_H=1}\|f_v\|_{\BMO(\real^d; H)}\,,$$
where $\BMO(\real^d; H)$ is the $H$-valued BMO-space on $\real^d$.  It is well known and easy to check that \cite[Theorem~IV.4.3]{Stein1993} holds equally for the Hilbert-valued case. We then deduce the previous lemma, plus its reciprocal. Let us record this explicitly as follows (recalling that $\mathrm{R}_d$ is the Hilbert space defined by \eqref{Rd}):

\begin{thm}\label{CarlsonlessBMObis}
 Let $f\in L_{\infty}(\M; \mathrm{R}_d^c)$. Then $f\in\BMO^c(\real^d,\M)$ iff $d\mu(f)$ is an $\M$-valued Carleson measure on $\real^{d+1}_+$. Moreover, if this is the case, then $\|d\mu(f)\|_{\rm C}\approx \|f\|^2_{\BMO^c}$.
\end{thm}

We will also need the dual description of $\H_p^c$ for $1<p<2$ as a BMO type space. This is the so-called $\BMO^c_q$-space studied in \cite{Mei2007}, which is the function analogue of the martingale $\BMO_q^c$ of \cite{JX2003}.  Let $2<q\le \8$. Define $\BMO^c_q(\real^d,\M)$ to be the space of all $f\in L_q(\M;\mathrm{R}_d^c)$ such that
 $$\|f\|_{\BMO^c_q}=
     \Big\|\mathop{{\sup}^+}_{s\in B\subset \real^d}\frac{1}{|B|} \int_{B}|f(t)-f_{B}|^2 dt
 \Big\|^{\frac 1 2}_{L_{\frac{q}{2}}(\N)}<\8.$$
Note that the norm $\|{\sup}^+_{i} a_i \|_{\frac{q}{2}}$ is just an intuitive notation since the pointwise supremum does not make any sense in the noncommutative setting. This is the norm of the Banach space $L_{\frac{q}{2}}(\N;\ell_\8)$; we refer to \cite{Pisier1998, ju-doob, jx-erg} for more information. Here we need only the following fact (which can be taken as definition):  $f\in\BMO^c_q(\real^d,\M)$ iff
 \beq\label{bmo sup}
 \exists\, a\in L_{\frac{q}{2}}(\N)\textrm{ s.t. }
 \frac{1}{|B|} \int_{B}|f(t)-f_{B}|^2 dt\le a(s) \textrm{ for all } s\in B\textrm{ and for all balls } B\subset\real^d;
 \eeq
if this is the case, then
 $$\|f\|^2_{\BMO^c_q}=\inf\big\{\|a\|_{L_{\frac{q}{2}}(\N)}: a \textrm{ as above}\big\}.$$

With this in mind, one immediately sees that Lemma~\ref{CarlsonlessBMO} transfers to the present setting with almost the same proof. Thus we have the following result whose proof is left to the reader.

\begin{lem}\label{CarlsonlessBMOq}
 Let $f\in \BMO_q^c(\real^d,\M)$ and $a$ satisfy \eqref{bmo sup}.  Then  $d\mu(f)$ is a $q$-Carleson measure in the following sense:
 $$\frac{1}{|B|}\int_{T(B)}|\Phi_\e*f (t)|^2\frac{dt d\e}{\e}\le a \textrm{ for all } s\in B\textrm{ and for all balls } B\subset\real^d.$$
\end{lem}

Like in the BMO case, the converse inequality holds too. We state this as the following theorem and postpone its proof to the next section.

\begin{thm}\label{CarlsonlessBMOqbis}
 Let $2<q\le \8$ and  $f\in L_{q}(\M; \mathrm{R}_d^c)$. Then $f\in\BMO_q^c(\real^d,\M)$ iff $d\mu(f)$ is a $q$-Carleson measure:
 $$
 \Big\| \mathop{{\sup}^+}_{s\in B\subset \real^d} \frac{1}{|B|}\int_{T(B)}|\Phi_\e*f (t)|^2\frac{dt d\e}{\e}
 \Big\|^{\frac 1 2}_{L_{\frac{q}{2}}(\N)}<\8.$$
\end{thm}


\section{Proof of Theorem \ref{equivalence Hp}}
\label{Proof1}


This section is mainly devoted to the proof of Theorem \ref{equivalence Hp}, which is the crucial part of the whole paper.  We will prove Theorem \ref{CarlsonlessBMOqbis} too.  Recall that $\Phi$ is of vanishing mean and satisfies the condition \eqref{schwartz}, and that the pair $(\Phi, \Psi)$ is fixed as in \eqref{reproduce}.

The proof of Theorem \ref{equivalence Hp} is long and technical. We will divide its main steps into several lemmas.

\begin{lem}\label{S-dual}
 Let $1\le p<2$ and $q$ be its conjugate index. Then for $f\in \H_p^c(\real^d,\M)\cap L_2(\N)$ and $g\in\BMO_q^c(\real^d,\M)$
 $$
 \Big|\tau \int _{\real^d} f(s)g^*(s)ds \Big|\les \|S^c_\Phi(f)\|_p\|g\|_{\BMO_q^c}\,.
 $$
 \end{lem}

\begin{proof}
 Let $f\in \H^c_p(\real^d,\M)$ with compact support (relative to the variable of $\real^d$). We assume that $f$ is sufficiently nice so that all calculations below are legitimate. Given $s\in\real^d$ and $r>0$, let  $B(s,r)$ denote the ball with center $s$ and radius $r$. We require two auxiliary square functions:
   \beq\label{truncated S}
   \left \{ \begin{split}
  \displaystyle&S^c_\Phi(f)(s,\e) =\Big(\int_\e^\8\int_{B(s,r-\frac{\e}2)} |\Phi_r*f (t)|^2\frac{dt d r}{r^{d+1}}\Big)^{\frac12}\\
  \displaystyle&\overline{S}^c_\Phi(f)(s,\e) =\Big(\int_\e^\8\int_{B(s,\frac{r}2)} |\Phi_r*f (t)|^2\frac{dt d r}{r^{d+1}}\Big)^{\frac12}
   \end{split} \right.
    \eeq
for $s\in\real^d$ and $\e>0$.
Both $S^c_\Phi(f)(s,\e)$ and $\overline{S}^c_\Phi(f)(s,\e)$ are decreasing in $\e$, $S^c_\Phi(f)(s,0) =S_{\Phi}^c(f) (s)$ and $S^c_\Phi(f)(s,+\8)  =0$. On the other hand, it is clear that  $\overline{S}^c_\Phi(f)(s,\e)^2 \le S^c_\Phi(f)(s,\e)^2$.  For notational simplicity, we will denote $S^c_\Phi(f)(s,\e)$ and $\overline{S}^c_\Phi(f)(s,\e)$ simply by $S(s,\e)$ and $\overline{S}(s,\e)$, respectively. By approximation, we can assume that $S(s,\e)$ and $\overline{S}(s,\e)$ are invertible for every $(s,\e)\in\real^{d+1}_+$. By \eqref{reproduce}, \eqref{Plancherel} and the Fubini theorem, we have
 \be\begin{split}
  \tau\int_{\real^d}f(s)g^*(s)ds
 &=\tau\int_{\real^{d+1}_+}\Phi_\e*f(s)\cdot \big(\Psi_\e*g(s)\big)^*\,\frac{ds d\e}{\e}\\
 &=\frac{2^d}{c_d}\,\tau\int_{\real^{d+1}_+}\int_{B(s, \frac\e2)}\Phi_\e*f(t)\cdot \big(\Psi_\e*g(t)\big)^*\,\frac{dtd\e}{\e^{d+1}}ds\\
 &=\frac{2^d}{c_d}\,\tau\int_{\real^{d+1}_+}\int_{B(s, \frac\e2)}\Phi_\e*f(t)S(s, \e)^{\frac{p-2}2}\cdot
 S(s, \e)^{\frac{2-p}2} \big(\Psi_\e*g(t)\big)^*\,\frac{dtd\e}{\e^{d+1}}ds,
 \end{split}\ee
where  $c_d$ is the volume of the unit ball of $\real^{d}$. Then by the Cauchy-Schwarz inequality,
  \be\begin{split}
  &\frac{c_d^2}{4^d}\,\Big|\tau \int _{\real^d} f(s)g^*(s)ds \Big|^2 \\
 &\les \tau\int_{\real^{d}}\int_0^\8 S(s,\e )^{p-2}
 \Big(\int_{B(s, \frac\e2)} |\Phi_\e*f (t)|^2 \,\frac{dt}{\e^{d+1}}\Big)
 d\e ds  \\
 & \,\;\cdot \tau\int_{\real^{d}}\int_0^\8 S(s,\e )^{2-p}
 \Big(\int_{B(s, \frac\e2)} |\Psi_\e*g (t)|^2 \,\frac{dt}{\e^{d+1}}\Big)
 d\e ds \\
 &\;{\mathop =^{\rm def}}\; {\rm A} \cdot {\rm B}.
  \end{split}\ee
To estimate the term A, using 
 $\overline{S}(s,\e)^2\le S(s,\e)^2$
and  $1\le p<2$, we get
  $$S(s,\e)^{\frac{p-2}2}\le\overline{S}(s,\e)^{\frac{p-2}2}.$$
 Therefore,
  \be\begin{split}
    {\rm A}
  &\le \tau\int_{\real^{d}}\int_0^\8 \overline{S}(s,\e )^{p-2} \Big(\int_{B(s, \frac\e2)} |\Phi_\e*f (t)|^2 \,\frac{dt}{\e^{d+1}}\Big) d\e ds \\
  &=-\tau\int_{\real^{d}}\int_0^\8 \overline{S}(s,\e )^{p-2} \frac{\partial}{\partial\e} \overline{S}(s,\e )^{2}d\e ds\\
 &=-2\tau\int_{\real^{d}}\int_0^\8 \overline{S}(s,\e )^{p-1} \frac{\partial}{\partial\e} \overline{S}(s,\e )d\e ds .
 \end{split}\ee
 Since $1\le p<2$ and $\overline{S}(s,\e )$ is decreasing in $\e$, $\overline{S}(s,\e )^{p-1}\le \overline{S}(s,0 )^{p-1}$.  On the other hand,  $- \frac{\partial}{\partial\e} \overline{S}(s,\e )\ge0$. Thus
  $$
  {\rm A}
  \le -2\tau\int_{\real^{d}}\overline{S}(s,0 )^{p-1}\int_0^\8 \frac{\partial}{\partial\e} \overline{S}(s,\e )d\e ds
  =2\tau\int_{\real^{d}}\overline{S}(s,0 )^{p}  ds\le 2\|S_\Phi^c(f)\|_{p}^p\,.
 $$

The estimate of B is  harder. For $j\in\ent$ we use the partition of $\real^d$ into dyadic cubes with side length $2^j$. Each such cube is of the form $Q_{m,j}=((m_1-1)2^{j},\, m_1 2^j] \times \cdots \times((m_d-1)2^{j}, \,m_d 2^j]$ with $m=(m_1,\cdots, m_d)\in\ent^d$. Let $c_{m,j}$ be its center. Define
 \beq\label{squareSS}
 \mathbb S(s, j)=\Big(\int_{\sqrt{d}\,2^j}^\8\int_{B(c_{m,j}, r)} |\Phi_r*f (t)|^2 \,\frac{dt dr}{r^{d+1}}\Big)^{\frac12}\;\;\textrm{ if }\;\;
 s\in Q_{m,j},
 \eeq
Since $B(s, r-\frac\e2)\subset B(c_{m,j}, r)$ whenever $s\in Q_{m,j}$ and $r\ge\e\ge \sqrt{d}\,2^j$,  we have
 $$S(s, \e)^2\le \mathbb S(s, j)^2 \;\textrm{ for }\;
 s\in Q_{m,j}\text{ and } \e\ge\sqrt{d}\,2^j.$$
Consequently,
 $$S(s, \e)^{2-p}\le \mathbb S(s, j)^{2-p}\;\textrm{ for }\;
 s\in Q_{m,j}\text{ and } \e\ge\sqrt{d}\,2^j.$$
Therefore,
  \be\begin{split}
  {\rm B}
  &=\tau\sum_m\sum_j\int_{Q_{m,j}}\int_{\sqrt{d}\,2^j}^{\sqrt{d}\,2^{j+1}}S(s,\e )^{2-p}
 \Big(\int_{B(s, \frac\e2)} |\Psi_\e*g (t)|^2 \,\frac{dt}{\e^{d+1}}\Big)d\e ds \\
 &\le \tau\sum_m\sum_j \int_{Q_{m,j}}\int_{\sqrt{d}\,2^j}^{\sqrt{d}\,2^{j+1}}\mathbb S(s, j)^{2-p}
 \Big(\int_{B(s, \frac\e2)} |\Psi_\e*g (t)|^2 \,\frac{dt}{\e^{d+1}}\Big)d\e ds \\
 &= \tau\int_{\real^d}\sum_j \mathbb S(s, j)^{2-p} \int_{\sqrt{d}\,2^j}^{\sqrt{d}\,2^{j+1}}
 \Big(\int_{B(s, \frac\e2)} |\Psi_\e*g (t)|^2 \,\frac{dt}{\e^{d+1}}\Big)d\e ds \\
 &= \tau\int_{\real^d}\sum_j\sum_{k\ge j} d(s, k) \int_{\sqrt{d}\,2^j}^{\sqrt{d}\,2^{j+1}}
 \Big(\int_{B(s, \frac\e2)} |\Psi_\e*g (t)|^2 \,\frac{dt}{\e^{d+1}}\Big)d\e ds ,\\
  \end{split}\ee
where $d(s, k)=\mathbb S(s, k)^{2-p} -\mathbb S(s, k+1)^{2-p}$. Since $\mathbb S(s, k)$ is decreasing in $k$ and $0<2-p\le1$, $d(s, k)\ge0$. On the other hand, $d(\cdot, k)$ is constant on $Q_{m, k}$ for all $m\in\ent^d$. Then we get
 \be\begin{split}
 {\rm B}
 &\le  \tau\int_{\real^d}\sum_{k} d(s, k)\sum_{j\le k} \int_{\sqrt{d}\,2^j}^{\sqrt{d}\,2^{j+1}}
 \Big(\int_{B(s, \frac\e2)} |\Psi_\e*g (t)|^2 \,\frac{dt}{\e^{d+1}}\Big)d\e ds \\
 &\le  \tau\sum_m\sum_{k} d(s, k) \int_{Q_{m, k}} \int_{0}^{\sqrt{d}\,2^{k+1}}
 \Big(\int_{B(s, \frac\e2)} |\Psi_\e*g (t)|^2 \,\frac{dt}{\e^{d+1}}\Big)d\e ds.
  \end{split}\ee
Since $g\in \BMO_q^c(\real^d,\M)$, Lemma~\ref{CarlsonlessBMOq} ensures the existence of a positive operator $a\in L_{\frac q2}(\N)$ such that $\|a\|_{\frac q2}\les \|g\|_{\BMO_q^c}^2$ and
 $$\frac1{|B|}\int_{T(B)} |\Psi_\e*g (t)|^2 \,\frac{dtd\e}{\e}\le a(s) \textrm{  for all  } s\in B
 \textrm{ and for all balls } B.$$
Let $B_{m, k}$ be the ball with center  $c_{m, k}$ and radius $\sqrt{d}\,2^{k+1}$.  Thus by the Fubini theorem,
 $$
 \int_{Q_{m, k}} \int_{0}^{\sqrt{d}\,2^{k+1}}
 \Big(\int_{B(s, \frac\e2)} |\Psi_\e*g (t)|^2 \,\frac{dt}{\e^{d+1}}\Big)d\e ds
 \les \int_{T(B_{m, k})} |\Psi_\e*g (t)|^2 \,\frac{dtd\e}{\e}
 \le \int_{Q_{m, k}}a(s)ds.
 $$
 Therefore, by the H\"older inequality,
 \be\begin{split}
 {\rm B}
 &\les  \tau\sum_m\sum_{k} d(s, k)\int_{Q_{m, k}}a(s)ds
 =\tau\sum_m\sum_{k} \int_{Q_{m, k}}d(s, k)a(s)ds\\
 &=\tau\int_{\real^d}\sum_{k} d(s, k)a(s)ds
 =\tau\int_{\real^d}\mathbb S(s, -\8)^{2-p} a(s)ds\\
 &=\tau\int_{\real^d}S^c_{\Phi}(f)(s)^{2-p} a(s)ds
 \le\|S^c_{\Phi}(f)\|_p^{2-p}\,\|a\|_{\frac q2}\\
 &\les \|S^c_{\Phi}(f)\|_p^{2-p}\,\|g\|_{\BMO_q^c}^2\,.
 \end{split}\ee
Combining the estimates of A and B, we finally get the desired inequality of the lemma.
 \end{proof}

We will need a  variant of the previous lemma. For any function $f$ defined on $\real^{d+1}_+$ with values in $ L_1(\M)+L_\8(\M)$, define (recalling that $\Ga=\{(t, \e)\in \real^{d+1}_+: |t|<\e\}$)
  \beq\label{T-squre}
  \mathcal{S}^c(f)(s)=\Big(\int_{\Gamma}|f(t+s,\e)|^2 \frac{dt d\e}{\e^{d+1}}\Big)^{\frac12} \,,\quad s\in\real^d.
  \eeq

\begin{lem}\label{T-dual}
 Let $1\le p<2$ and $q$ be its conjugate index. Then for any compactly supported function $f$ on $\real^{d+1}_+$ with values in $ L_1(\M)\cap L_\8(\M)$ and $g\in\BMO_q^c(\real^d,\M)$
 $$
 \Big|\tau\int_{\real^{d+1}_+}f(s,\e)\cdot \big(\Psi_\e*g(s)\big)^*\,\frac{ds d\e}{\e}\Big|\les \|\mathcal{S}^c(f)\|_p\,\|g\|_{\BMO_q^c}\,.
 $$
 \end{lem}
 
\begin{proof}
 This proof is exactly the same as that of the previous lemma, just by replacing the function $(s,\e)\mapsto \Phi_\e*f(s)$ in that proof by $f$ in the present lemma.
\end{proof}

We will also need the radial version of Lemma~\ref{S-dual}. To this end, we have to control the radial square function by the conic one. For the classical  Littlewood-Paley $g$-function and Lusin area integral, this fact follows simply from the harmonicity of the Poisson integral.
Since the harmonicity is no longer available, the proof of our inequality is more elaborated.  Compared with \cite{Mei2007}, this is a new phenomenon which seems new even going back to the commutative case. We will use multi-index notation.  For $m=(m_1,\cdots, m_d)\in \mathbb{N}_0^d$ ($\nat_0$ being the set of nonnegative integers) and $s=(s_1,\cdots, s_d)\in \real^d$, we set $s^m=s_1^{m_1}\cdots s_d^{m_d}$. Let $|m|_1=m_1+\cdots +m_d$ and
 $$D^m=\frac{\partial^{m_1}}{\partial s_1^{m_1}}\cdots \frac{\partial^{m_d}}{\partial s_d^{m_d}}.$$

\begin{lem}\label{square1less2}
  Let $f\in L_1(\M; \mathrm{R}_d^c)+L_{\infty}(\M; \mathrm{R}_d^c)$. Then
 $$
 s^c_{\Phi}(f)(s)^2\les \sum_{|m|_1\leq d}S^c_{D^m \Phi}(f)(s)^2,\quad \forall s\in \real^d.
 $$
\end{lem}

\begin{proof}
Without loss of generality, we assume that  $f$ is selfadjoint and $\Phi$ is real-valued. Fix a point $s_0$, say $s_0=0$. For any $t \in \Gamma$, successive  applications of integration by parts yield (with $\partial_r=\frac{\partial}{\partial_r}$)
 \be\begin{split}
&|\Phi_\e*f(t)|^2 - |\Phi_\e*f(0)|^2 = \int _{0}^1 (r)'\partial_r( |\Phi_\e *f(rt )|^2)dr\\
&= \partial_r (|(\Phi_\e)*f(rt)|^2)\bigg|_{r=1} -  \int _{0}^1 r\partial^2_r( |(\Phi_\e)*f(r t )|^2) dr \\
& ~~~~\vdots\\
&=\sum_{j=1}^k \frac{(-1)^{j-1}}{j!} \partial^j_r (|\Phi_\e *f(rt)|^2)\bigg|_{r=1}+ \frac{(-1)^k}{k!} \int _{0}^1 r^k \partial^{k+1}_r( |\Phi_\e *f(rt)|^2) dr.
\end{split}\ee
For each derivative of order less than or equal to $k$ on  the right-hand side, we have
\be\begin{split}
\partial^j_r (|(\Phi_\e)*f(rt)|^2)
& = \sum_{i=0}^j  \begin{pmatrix} j \\ i\end{pmatrix} \partial^i_r [\Phi_\e*f(rt)]\cdot \partial^{j-i}_r [\Phi_\e*f(rt)]\\
&= \sum_{i=0}^j  \begin{pmatrix} j \\ i\end{pmatrix} \sum_{|m|_1=i} \frac{t^m}{\e^{|m|_1}}\, (D^m \Phi)_\e*f(rt)
\cdot \sum_{|n|_1=j-i}\frac{t^n}{\e^{|n|_1}}\,  (D^n\Phi)_\e*f(rt).
\end{split}\ee
Since $|t^m|\leq |t|^{|m|_1}\leq \e^{|m|_1}$ whenever $|t|\le\e$, using the inequality $ab+ba\le a^2+b^2$ for selfadjoint operators $a$ and $b$, we get
\be
 \partial^j_r (|(\Phi_\e)*f(r t)|^2)\bigg|_{r=1}
 \leq  \sum_{i=0}^j  \begin{pmatrix} j \\ i\end{pmatrix} \sum_{|m|_1=i}\, \sum_{|n|_1=j-i} \frac{1}{2}
  \big[\big|(D^m \Phi)_\e*f(t)\big|^2+ \big|(D^n\Phi)_\e*f(t)\big|^2\big].
\ee
On the other hand, to deal with the last derivative of order $k+1$, we use the following similar estimate:
 \be\begin{split}
 &\partial^{k+1}_r (|(\Phi_\e)*f(r t)|^2)\\
 &\leq \frac{|t|^{k+1}}{\e^{k+1}} \sum_{i=0}^{k+1}  \begin{pmatrix} k+1 \\ i\end{pmatrix} \sum_{|m|_1=i}\, \sum_{|n|_1=k+1-i} \frac{1}{2}
  \big[\big|(D^m \Phi)_\e*f(rt)\big|^2+ \big|(D^n\Phi)_\e*f(rt)\big|^2\big].
 \end{split}\ee
Thus for $|t|\leq \e$, we have
\be\begin{split}
&\int _{0}^1 r^k \partial^{k+1}_r ( |\Phi_\e *f(rt )|^2) dr  \\
&\leq \int _{0}^{\e }\frac{r^{k}}{\e^{k+1}} \sum_{i=0}^{k+1}  \begin{pmatrix} k+1 \\ i\end{pmatrix}
 \sum_{|m|_1=i} \sum_{|n|_1=k+1-i} \frac{1}{2}
 \big[\big|(D^m \Phi)_\e*f(\frac{rt}{|t|})\big|^2+ \big|(D^n\Phi)_\e*f(\frac{rt}{|t|})\big|^2\big]\,dr.
 \end{split}\ee
Letting $k=d-1$ and combining the previous inequalities, we get
\be\begin{split}
|\Phi_\e *f(0)|^2
\les  \sum_{|m|_1<d} \big|(D^m \Phi)_\e *f(t )\big|^2 +\sum_{|m|_1\le d} \int _{0}^{\e }\frac{r^{d-1}}{\e^{d}}
  \sum_{|m|_1\le d}\big|(D^m \Phi)_\e*f(\frac{rt}{|t|})\big|^2\,dr.
\end{split}\ee
Now divide by $\varepsilon^{d+1}$ both sides of the above inequality, then take  integration in $(t,\e)$ on $\Gamma$. The result for the left hand side is
 \be\begin{split}
 \int_{\Gamma}|f*\Phi_\e (0)|^2\frac{dtd\e }{\e ^{d+1}}= c_d\int_0^\8  |f*\Phi_\e (0)|^2 \frac{d\e }{\e }
 =c_d\,s_\Phi^c(f)(0)^2.
 \end{split}\ee
The one for the first sum on the right hand side is  equal to the sum of  $\big(S^c_{D^m \Phi}( f)(0) \big)^2$ for all multi-indices $m$ with $|m|_1<d$. As far as for the second sum,  an easy calculation yields
 \be\begin{split}
\int_{\Gamma}\int _{0}^{\e }\frac{r^{d-1}}{\e ^{d}} \big|(D^m \Phi)_\e *f(\frac{rt}{|t|})\big|^2dr\, \frac{dtd\e }{\e ^{d+1}}
=\frac{1}{d}\int_{\Gamma}\big| (D^m \Phi)_\e *f(t)\big|^2  \frac{dtd\e }{\e ^{d+1}}
=\frac{1}{d}\,  S^c_{D^m \Phi}( f)(0) ^2.
\end{split}\ee
Therefore, we have proved the announced assertion.
\end{proof}

 \begin{lem}\label{s-dual}
 Keep the assumption of Lemma~\ref{S-dual}. Then
 $$
 \Big|\tau \int _{\real^d} f(s)g^*(s)ds \Big|\les \|s^c_\Phi(f)\|_p^{\frac p 2}\,\|f\|_{\H^c_p}^{1-\frac p2}\,\|g\|_{\BMO_q^c}\,.
 $$
 \end{lem}

\begin{proof}
 This proof is similar to that of Lemma~\ref{S-dual}. Now consider the truncated version of $s^c_\Phi(f)$:
 $$
 s^c_\Phi(s,\e) =\Big(\int_\e^\8 |\Phi_r*f (s)|^2\frac{d r}{r}\Big)^{\frac12}\,.
 $$
As in the previous proof, we have
  \be\begin{split}
  &\Big|\tau \int _{\real^d} f(s)g^*(s)ds \Big|^2 \\
 &\le \tau\int_{\real^{d}}\int_0^\8 s^c_\Phi(s,\e )^{p-2}|\Phi_\e*f (s)|^2 \,\frac{ d\e ds}{\e}
 \cdot \tau\int_{\real^{d}}\int_0^\8 s^c_\Phi(s,\e )^{2-p}|\Psi_\e*g (s)|^2 \,\frac{ d\e ds}{\e}\\
 &\;{\mathop = ^{\rm def}}\;{\rm A}' \cdot {\rm B}'.
  \end{split}\ee
The term ${\rm A}' $ is estimated exactly as before,  so
  $
  {\rm A}'
  \le 2\|s_\Phi^c(f)\|_{p}^p\,.
 $

To estimate ${\rm B}'$, we note that the proof of Lemma~\ref{square1less2} also gives
 $$
  s^c_{\Phi}(f)(s,\e)^{2}\les \sum_{|m|_1\le d}S^c_{D^m \Phi}(f)(s, \e)^2,
  $$
where $S^c_{D^m \Phi}(f)(s, \e)$ is the truncation of $S^c_{D^m \Phi}(f)(s)$ as defined in \eqref{truncated S} with $D^m \Phi$ instead of $\Phi$.
Then
 $${\rm B}'\les \sum_{|m|_1\le d} \tau\int_{\real^{d}}\int_0^\8 S^c_{D^m \Phi}(f)(s,\e )^{2-p}|\Psi_\e*g (s)|^2 \,\frac{ds d\e}{\e}\,.$$
All terms on $S^c_{D^m \Phi}$ are handled in the same way, so it suffices to consider $S^c_{\Phi}$  (i.e., without derivation). Starting from this point, the reasoning becomes the same as for B before, except that in the final step, we invoke lemma~\ref{SboundedH1toL1}. Thus we conclude that
 $${\rm B}' \les\sum_{|m|_1\le d} \|S^c_{D^m \Phi}(f)\|_p^{2-p}\,\|g\|_{\BMO_q^c}^2 \les \|f\|_{\H^c_p}^{2-p}\,\|g\|_{\BMO_q^c}^2\,.$$
This finishes the proof of the lemma.
 \end{proof}

Another lemma will be needed for  the proof of Theorem \ref{equivalence Hp}. Recall that for $f:\real^{d+1}_+\to L_1(\M)+L_\8(\M)$, the square function $\mathcal{S}^c(f)$ is defined by \eqref{T-squre}. Now define
 $$T_p^c=\big\{f: \mathcal{S}^c(f)\in L_p(\N)\big\}, \;\text{ equipped with } \; \|f\|_{T^c_p}=\|\mathcal{S}^c(f)\|_p\,.$$
This is the column tent space already considered in \cite{Mei2007}.  $T^c_p$ is viewed as a subspace of the column space $L_p(\N;L_2^c(\Gamma))$ by the injection $f\mapsto\wt f$, where $\wt f(s, t, \e)=f(s+t,\e)$.  Note that the elements of  $L_p(\N;L_2^c(\Gamma))$ are considered as functions of three variables $(s, t,\e)$ with $s\in\real^d$ and $(t,\e)\in\Gamma$.
Then it is easy to show that the orthogonal projection from $L_2(\N;L_2^c(\Gamma))$ onto $T^c_2$ is given by the following
 $$P(F)(s,\e)=\frac{1}{|B(0,\e)|}\int_{B(0,\e)}F(s-u,u,\e)du.$$

\begin{lem}\label{tent-complemented}
The above projection $P$ extends to a bounded projection from $L_p(\N;L_2^c(\Gamma))$  onto $T^c_p$ for any $1<p<\8$.
\end{lem}

\begin{proof}
 We need only to consider the case $p>2$. Fix $F\in L_p(\N; L_2^c(\Gamma))$.  Denote $r$ the conjugate number of $\frac p2$, and choose a function $g\in L_r(\N)$ with norm one such that
 $$\|P(F)\|^2_{T^c_p}=\tau\int_{\real^d}\mathcal{S}^c(P(F))(s)^2g(s)ds=\tau\int_{\real^d}\int_0^\8\int_{B(0,\e)}|P(F)(s+t,\e)|^2\frac{dtd\e}{\e^{d+1}}\,g(s)ds.$$
Then by \eqref{CS}, two changes of variables  and  the Fubini theorem, we get
\be\begin{split}
 \|P(F)\|^2_{T^c_p}
 &\leq \tau  \int_{\real^d}\int_0 ^\8\int_{|t|<\e} \frac{1}{|B(0,\e)|} \int_{|u|<\e} | F(s+t-u,u,\e)|^2du\, \frac{dtd\e}{\e^{d+1}} \,|g(s)|\,ds \\
 &=\tau \int_{\real^d}\int_0 ^\8  \frac{1}{|B(0,\e)|} \int_{|s+t-u|<\e}\int_{|t|<\e}| F(u,t,\e)|^2du\, \frac{dtd\e}{\e^{d+1}} \,|g(s)|\,ds \\
 &\leq \tau \int_{\real^d} \int_0^\8  \int_{|t|<\e}| F(u,t,\e)|^2 \,\frac{dtd\e}{\e^{d+1}} \,\frac{1}{|B(0,\e)|} \int_{|s-u|<2\e} |g(s)| \,ds \,du\\
 &=2^d \tau \int_{\real^d} \int_{\Gamma}| F(u,t,\e)|^2 \,\frac{dtd\e}{\e^{d+1}} \,\frac{1}{|B(0,2\e)|} \int_{|s-u|<2\e} |g(s)| \,ds \,du.
 \end{split}\ee
By Mei's noncommutative Hardy-Littlewood maximal inequality (see \cite[Theorem 3.3]{Mei2007}), we find a positive operator $a\in L_r(\N)$ such that 
 $$\|a\|_{ r}\les\|g\|_{ r}\les 1\;\text{ and }\;\frac{1}{|B(0,2\e)|} \int_{|s-u|<2\e} |g(s)| \,ds\leq a(u),\;\forall u\in\real^d,\, \forall \e>0.$$
Thus by the H\"older inequality,
 $$
 \|P(F)\|^2_{T^c_p}\le 2^d \tau \int_{\real^d} \int_{\Gamma}| F(u,t,\e)|^2 \,\frac{dtd\e}{\e^{d+1}} \,a(u)  \,du\les\|F\|_{L_p(\N;L_2^c(\Gamma))}^2\,,
 $$
which finishes the proof of the lemma.
 \end{proof}

 \begin{rk}
  The previous lemma shows the duality equality $\big(T^c_p\big)^*=T^c_q$ for any $1<p<\8$ ($q$ being conjugate to $p$). This result was already observed in \cite{Mei2007}(see the remark following Theorem~4.7 there).
 \end{rk}
 
We are now ready to prove Theorem~\ref{equivalence Hp}.

\begin{proof}[Proof of Theorem~\ref{equivalence Hp}]
 The case $p=2$ is trivial. Consider now the case $1\le p<2$. Both majorations are contained in Lemma~\ref{SboundedH1toL1}. On the other hand, taking the supremum on the left hand side of the inequality in Lemma~\ref{S-dual} over $g$ in the unit ball of $\BMO^c_q(\real^d,\M)$ and involving Mei's duality theorem (see \cite[Theorem~4.4]{Mei2007}), we get
  $$\|f\|_{\H_p^c}\les \|S_{\Phi}^c(f)\|_{p}$$
for $f\in\mathcal{H}^c_p(\real^d,\M)\cap L_2(L_\8(\real^d)\overline\ot\M)$. Then a density argument shows that the same inequality also holds for all $f\in\mathcal{H}^c_p(\real^d,\M)$. The inequality $\|f\|_{\H_p^c}\les\|s_{\Phi}^c(f)\|_{L_p}$ is proved in the same way by virtue of Lemma~\ref{s-dual}.

 Pass to the case $2<p<\8$.  Let $q$ be the conjugate index of $p$. Let $f\in\mathcal{H}^c_p(\real^d,\M)$ and choose $g\in \mathcal{H}^c_q(\real^d,\M)$ such that $\|g\|_{\H_q^c}=1$ and
  $$\|f\|_{\H_p^c}\approx \tau\int_{\real^d}f(s)g(s)^*ds
  =\tau\int_{\real^{d+1}_+}\Phi_\e*f(s)\cdot \big(\Psi_\e*g(s)\big)^*\,\frac{ds d\e}{\e}.$$
 Then by  the H\"older inequality and Lemma~\ref{SboundedH1toL1} (applied to $g$, $\Psi$ and $q$),
 $$
 \|f\|_{\H_p^c}\les\|s_\Phi^c(f)\|_p\,\|s_\Psi^c(g)\|_q
 \les\|s_\Phi^c(f)\|_p\,\|g\|_{\H_q^c}\le \|s_\Phi^c(f)\|_p\,.$$
Similarly,
  $$ \|f\|_{\H_p^c}\les \|S_\Phi^c(f)\|_p\,.$$

It remains to show the two reverse inequalities. It suffices to show the reverse inequality for the conic square function since the one for the radial square function will then follow from Lemma~\ref{square1less2}.
   Let $f\in\mathcal{H}^c_p(\real^d,\M)$ and choose $G\in L_q(\N;L_2^c(\Gamma))$ with norm one such that
 \be\begin{split}
  \|S^c_\Phi(f)\|_{p}
  &=\tau\int_{\real^d}\int_{\Gamma} \Phi_\e*f(s+t,\e) G(s, t,\e)^* \,\frac{dtd\e}{\e^{d+1}}\,ds\\
  &=\tau\int_{\real^d}\int_0^\8 \Phi_\e*f(s,\e) g(s,\e)^* \,\frac{dsd\e}{\e}\,,
   \end{split}\ee
  where $g=P(G)$. Now by Lemma~\ref{T-dual} with $f$ and $g$ exchanged (as well as $p$ and $q$), we deduce that
   $$\|S^c_\Phi(f)\|_{p}\les \|g\|_{T^c_q}\|f\|_{\BMO^c_p}\les \|G\|_{L_q(\N;L_2^c(\Gamma))}\|f\|_{\BMO^c_p}\les \|f\|_{\H_p^c}\,,$$
where we have used  Lemma~\ref{tent-complemented} and the equality $\BMO^c_p(\real^d,\M)=\mathcal{H}^c_p(\real^d,\M)$ for $2<p<\8$ (see \cite[Theorem~4.7]{Mei2007}). Therefore, the proof of the theorem is complete.
 \end{proof}

\begin{proof}[Proof of Theorem~\ref{CarlsonlessBMOqbis}]
 Reexamining the proof of Lemma~\ref{S-dual}, we realize that $\|g\|_{\BMO_q^c}$ in the inequality there can be replaced by the $q$-Carleson measure norm of $g$ associated to $\Psi$. Namely, we have
  \be\begin{split}
 \Big|\tau \int _{\real^d} f(s)g^*(s)ds \Big|
 &\les \|S^c_\Phi(f)\|_p\,
 \Big\|\mathop{{\sup}^+}_{s\in B\subset \real^d} \frac{1}{|B|}\int_{T(B)}|\Psi_\e*f (t)|^2\frac{dt d\e}{\e} \Big\|^{\frac 1 2}_{L_{\frac{q}{2}}(\N)}\\
 &\les \|f\|_{\H^c_p}\,
 \Big\|\mathop{{\sup}^+}_{s\in B\subset \real^d} \frac{1}{|B|}\int_{T(B)}|\Psi_\e*f (t)|^2\frac{dt d\e}{\e} \Big\|^{\frac 1 2}_{L_{\frac{q}{2}}(\N)}\,.
 \end{split}\ee
Now taking the supremum over $f$ in the unit ball of $\H_p^c(\real^d,\M)$ yields
 $$\|f\|_{\BMO_q^c}\les \Big\|\mathop{{\sup}^+}_{s\in B\subset \real^d} \frac{1}{|B|}\int_{T(B)}|\Psi_\e*f (t)|^2\frac{dt d\e}{\e} \Big\|^{\frac 1 2}_{L_{\frac{q}{2}}(\N)}\,.$$
 This is the desired reverse inequality (with $\Psi$ instead of $\Phi$).
  \end{proof}

We conclude this section with a result in the spirit of Lemma~\ref{tent-complemented} which is of independent interest.
Let $H=L_2((0,\8),\frac{d\e}{\e})$ and consider the space $L_p(\N; H^c)$. The elements of the latter space are viewed as functions defined on $\real^{d+1}_+$ with values in $L_p(\M)$. Let $\mathcal{T}_p$ be its closed subspace spanned by all functions of the form  $(s, \e)\mapsto \Phi_\e*f(s)$ with $f\in L_p(\N)$ such that $s^c_\Phi(f)\in L_p(\N)$. We now calculate the orthogonal projection $\mathsf T$ from $L_2(\N; H^c)$ onto $\mathcal{T}_2$. Let $F\in L_2(\N; H^c)$ and $g\in L_2(\N)$. Then (with $\wt\Phi(s)=\overline{\Phi(-s)}$)
  $$
  \la \mathsf T(F), \, \Phi_{\cdot}*g\ra
  =\tau\int_{\real^{d+1}_+}F(s, \e)(\Phi_\e*g(s))^*\,\frac{ds d\e}\e
  =\tau\int_{\real^{d}}f(s)\, g(s)^*ds,
 $$
where
 $$f(s)=\int_0^\8\wt\Phi_\e*F(\cdot, \e)(s)\,\frac{d\e}\e\,.$$
Let $\mathsf k_\Phi^c$ be the operator introduced before Lemma~\ref{SboundedH1toL1} and $(\mathsf k_\Phi^c)'$ its adjoint.  Then clearly,  $f=(\mathsf k_\Phi^c)'(F)$. On the other hand,
 $$\la \mathsf T(F), \, \Phi_{\cdot}*g\ra
 =\tau\int_{\real^{d+1}_+}\Psi_\e*f(s)\,(\Phi_\e*g(s))^*\,\frac{ds d\e}\e
 =\la \mathsf k_\Psi^c(f), \, \mathsf k_\Phi^c(g)\ra.$$
Note that $\mathsf k_\Psi^c(f)$ belongs to $\mathcal{T}_2$ since by the choice of $\Psi$ just after  \eqref{reproduce}: $\mathsf k_\Psi^c(f)=\mathsf k_\Phi^c(\zeta*f)$ with $\zeta$ given by $\wh\zeta=\frac{\eta}{ h}$. Thus
 $$\mathsf T(F)=\mathsf k_\Psi^c\circ (\mathsf k_\Phi^c)'(F).$$

\begin{prop}\label{projection}
The orthogonal projection $\mathsf T$ is bounded on $L_p(\N; H^c)$ for $1<p<\8$. A similar statement also holds for $H=L_2(\Ga,\frac{dtd\e}{\e^{d+1}})$ which corresponds to the conic square function.
\end{prop}

\begin{proof}
 It suffices to consider the case $p>2$. Let $F\in L_p(\N; H^c)$. Then by Theorem~\ref{equivalence Hp} and Corollary~\ref{CZHdual}, we have
  $$
  \big\|\mathsf T(F)\big\|_{L_p(\N; H^c)}=\big\|s^c_\Psi\big( (\mathsf k_\Phi^c)'(F)\big)\big\|_{L_p(\N)}
  \les \big\|(\mathsf k_\Phi^c)'(F)\big\|_{\H_p^c(\real^d,\M)}
   \les \big\|F\big\|_{\H_p^c(\real^d,H^c\overline\ot\M)}\,.
 $$
 However, 
  $$\big\|F\big\|_{\H_p^c(\real^d,H^c\overline\ot\M)}\les \big\|F\big\|_{L_p(\N; H^c)}\;\text{ for }2\le p<\8.$$
Therefore,
 $$\big\|\mathsf T(F)\big\|_{L_p(\N; H^c)}\les  \big\|F\big\|_{L_p(\N; H^c)}\,,$$
  as desired.
\end{proof}

\section{Proof of Theorem \ref{equivalence HpD}}
\label{Proof2}


In this section, the pair $(\Phi, \Psi)$ will be fixed as in \eqref{reproduceD}.
The proof of Theorem \ref{equivalence HpD} is similar to that of Theorem \ref{equivalence Hp}. We will be brief by indicating the necessary modifications. We first prove
the discrete counterparts of Lemmas \ref{S-dual}, \ref{T-dual} and \ref{s-dual}.
For any function $F:\real^d\times\ent\to L_1(\M)+L_\8(\M)$ define
 $$\mathcal{S}^{c, D}(F)=\Big(\sum_{k\in\ent} 2^{-dj}\int_{B(s,2^j)} |F (t, j)|^2 dt\Big)^{\frac12}.$$
Note that if $F(s, j)= \Phi_{2^j}*f(s)$ for some $f:\real^d\to L_1(\M)+L_\8(\M)$, then $\mathcal{S}^{c, D}(F)=S_\Phi^{c, D}(f)$, the latter being the discrete square function introduced after Theorem \ref{equivalence HpD}.  The following is the discrete analogue of Lemma~\ref{T-dual}.

\begin{lem}\label{S-dualD}
 Let $1\le p<2$ and $q$ be its conjugate index. Then for any compactly supported function $f:\real^d\times\ent\to L_1(\M)\cap L_\8(\M)$ and $g\in\BMO_q^c(\real^d,\M)$
 $$
 \Big|\tau\int_{\real^d}\sum_jf(s, j)\cdot \big(\Psi_{2^j} *g(s)\big)^*\,ds \Big|\les \|\mathcal{S}^{c,D}(f)\|_{L_p(\N)}\,\|g\|_{\BMO_q^c}\,.
 $$
 \end{lem}
\begin{proof}
 As in the continuous case, we require the  truncated version of  $\mathcal{S}_{\Phi}^{c,D}$: For $j\in\ent$ let
  \be\begin{split}
  &\mathcal{S}^{c,D}(f)(s,j)=\Big(\sum_{k\ge j} 2^{-dk}\int_{B(s,2^k-2^{j-1})} |f (t, k)|^2 dt\Big)^{\frac12}\,,\\
  &\overline{\mathcal{S}}^{c,D}(f)(s,j)=\Big(\sum_{k\ge j}2^{-dk}\int_{B(s,2^{k-1})} |f(t, k)|^2 dt\Big)^{\frac12}\,.
  \end{split}\ee
Denote $\mathcal{S}_{\Phi}^{c,D}(f)(s,j)$ and $\overline{\mathcal{S}}_{\Phi}^{c,D}(f)(s,j)$ simply by $\mathcal{S}(s,j)$ and $\overline{\mathcal{S}}(s,j)$, respectively. By approximation, we may assume that $\mathcal{S}(s,j)$  and $\overline{\mathcal{S}}(s,j)$ are invertible for every $s\in\real^d$ and $j\in\mathbb{Z}$.  By \eqref{reproduceD} and the Cauchy-Schwarz inequality,
 \be\begin{split}
 \Big|\tau &\int_{\real^d}\sum_j f(s,j)\cdot \big(\Psi_{2^j} *g(s)\big)^*\,ds \Big|^2\\
 &=\Big|\frac{2^d}{c_d}\,\tau\int_{\real^d}\sum_j 2^{-dj}\int_{B(s, 2^{j-1})} f(s,j)\cdot \big(\Psi_{2^j} *g(t)\big)^*\,dt\,ds  \Big|^2\\
 &\les \tau\int_{\real^{d}} \sum_j \mathcal{S}(s,j)^{p-2}
 \Big( 2^{-dj}\int_{B(s, 2^{j-1}))} |f(s,j)|^2 \,dt\Big)
  ds  \\
 & \,\;\cdot \tau\int_{\real^{d}} \sum_j \mathcal{S}(s,j)^{2-p}
 \Big( 2^{-dj}\int_{B(s, 2^{j-1})} |\Psi_{2^j}*g (t)|^2 \,dt\Big)
  ds \\
 &\;{\mathop =^{\rm def}}\; {\rm I} \cdot {\rm II}.
  \end{split}\ee
The term $\rm I$ is less easy to estimate than the corresponding term $\rm A$  in the proof of Lemma \ref{S-dual}. To deal with it we simply  set
$\overline{\mathcal{S}}_j=\overline{\mathcal{S}}(s,j)$ and $\overline{\mathcal{S}}=\overline{\mathcal{S}}(s,-\8)\le \mathcal{S}^{c,D}(f)(s)$. Then
 \be\begin{split}
 {\rm I}
 &\le \tau\int_{\real^{d}} \sum_j \overline{\mathcal{S}}_j^{p-2}(\overline{\mathcal{S}}_j^2-\overline{\mathcal{S}}_{j+1}^2)ds \\
 &=\tau\int_{\real^{d}} \sum_j \big[\overline{\mathcal{S}}_j^{p-1}(\overline{\mathcal{S}}_j-\overline{\mathcal{S}}_{j+1})+\overline{\mathcal{S}}_j^{p-2}\overline{\mathcal{S}}_{j+1}(\overline{\mathcal{S}}_j-\overline{\mathcal{S}}_{j+1})\big]ds.
  \end{split}\ee
Since $1\le p<2$, $\overline{\mathcal{S}}_j^{p-1}\le \overline{\mathcal{S}}^{p-1}$. So
 $$\tau\int_{\real^{d}} \sum_j \overline{\mathcal{S}}_j^{p-1}(\overline{\mathcal{S}}_j-\overline{\mathcal{S}}_{j+1})ds
 \le\tau\int_{\real^{d}} \overline{\mathcal{S}}^{p-1} \sum_j (\overline{\mathcal{S}}_j-\overline{\mathcal{S}}_{j+1})ds
 =\tau\int_{\real^{d}} \overline{\mathcal{S}}^{p}.$$
On the other hand,
 $$
 \tau\int_{\real^{d}} \sum_j\overline{\mathcal{S}}_j^{p-2}\overline{\mathcal{S}}_{j+1}(\overline{\mathcal{S}}_j-\overline{\mathcal{S}}_{j+1})ds
 =\tau\int_{\real^{d}} \sum_j \overline{\mathcal{S}}^{\frac{1-p}2}\overline{\mathcal{S}}_j^{p-2}\overline{\mathcal{S}}_{j+1}\overline{\mathcal{S}}^{\frac{1-p}2}\cdot \overline{\mathcal{S}}^{\frac{p-1}2}(\overline{\mathcal{S}}_j-\overline{\mathcal{S}}_{j+1})\overline{\mathcal{S}}^{\frac{p-1}2}ds.
 $$
However,
 $$\overline{\mathcal{S}}^{\frac{1-p}2}\overline{\mathcal{S}}_j^{p-2}\overline{\mathcal{S}}_{j+1}\overline{\mathcal{S}}^{\frac{1-p}2}
 =\overline{\mathcal{S}}^{\frac{1-p}2}\overline{\mathcal{S}}_j^{\frac{p-1}2}\cdot \overline{\mathcal{S}}_j^{\frac{p-3}2}\overline{\mathcal{S}}_{j+1}^{\frac{3-p}2}\cdot
 \overline{\mathcal{S}}_{j+1}^{\frac{p-1}2}\overline{\mathcal{S}}^{\frac{1-p}2}.$$
Note that each of the three factors on the right-hand side is a contraction. Consider, for instance, the first one:
 $$\overline{\mathcal{S}}^{\frac{1-p}2}\overline{\mathcal{S}}_j^{\frac{p-1}2}\big[\overline{\mathcal{S}}^{\frac{1-p}2}\overline{\mathcal{S}}_j^{\frac{p-1}2}\big]^*
 =\overline{\mathcal{S}}^{\frac{1-p}2}\overline{\mathcal{S}}_j^{p-1}\overline{\mathcal{S}}^{\frac{1-p}2}\le \overline{\mathcal{S}}^{\frac{1-p}2}\overline{\mathcal{S}}^{p-1}\overline{\mathcal{S}}^{\frac{1-p}2}= 1.$$
Therefore, by the H\"older inequality,
 $$
 \tau\int_{\real^{d}} \sum_j\overline{\mathcal{S}}_j^{p-2}\overline{\mathcal{S}}_{j+1}(\overline{\mathcal{S}}_j-\overline{\mathcal{S}}_{j+1})ds
 \le\tau\int_{\real^{d}} \sum_j \overline{\mathcal{S}}^{\frac{p-1}2}(\overline{\mathcal{S}}_j-\overline{S}_{j+1})\overline{\mathcal{S}}^{\frac{p-1}2}ds
 =\tau\int_{\real^{d}} \overline{\mathcal{S}}^{p}.
 $$
Combining the preceding inequalities, we get the desired estimate of I:
 $${\rm I}\le 2\tau\int_{\real^{d}} \overline{\mathcal{S}}^{p}\le2 \|\mathcal{S}^{c, D}(f)\|_p^{p}\,.$$

The estimate of the term $\rm II$ is, however,  almost identical to that of  $\rm B$ in the proof of Lemma \ref{S-dual}. There exist only two minor differences. The first one concerns the square function  $\mathbb S(s, j)$ in \eqref{squareSS}: it is now replaced by
$$
 \mathbb S(s, j)=\Big(\sum_{k\ge j+j_0}2^{-dk}\int_{B(c_{m,j}, 2^k)} |f(t, k)|^2 \,dt\Big)^{\frac12}\;\;\textrm{ if }\;\;
 s\in Q_{m,j},
 $$
where $j_0$ is the smallest integer such that $2^{j_0}\ge \sqrt d$.  Then we have $\mathcal{S}(s, j)\le \mathbb S(s, j)$. The second difference is about the Carleson characterization of $\BMO_q^c$ in Lemma~\ref{CarlsonlessBMOq}: we now use its discrete analogue. Namely,  for $g\in\BMO_q^c(\real^d,\M)$ define
 $$d\mu_D(g) = \sum_{j=-\8}^\8  |\Psi_{2^j}*g(s)|^2 ds \times d\delta_{2^j}(\e) ,$$
where $\delta_{2^j}(\e)$ is the unit Dirac mass at the point $2^j$, considered as a measure on $\real_+$.
Then $d\mu_D(g)$ is a $q$-Carleson measure on $\real^{d+1}_+$ and
 $$\Big\|\mathop{{\sup}^+}_{s\in B\subset \real^d}\frac{1}{|B|}
 \int_{T(B)}\sum_{j=-\8}^\8  |\Psi_{2^j}*g (s)|^2 ds \times d\delta_{2^j}(\e) \Big\|_{L_{\frac{q}{2}}(\N)}\les\|g\|^2_{\BMO_q^c}\,.$$
The proof of this property is the same as  that of  Lemma~\ref{CarlsonlessBMOq}.   Except these two differences, the remainder of the argument for  $\rm II$ is identical with that for B in the proof of Lemma \ref{square1less2}.  Thus we conclude  that
$${\rm II}\les\|\mathcal{S}^{c, D}(f)\|_p^{2-p}\,\|g\|_{\BMO_q^c}^2\,.$$
Hence the lemma is proved.
 \end{proof}

 \begin{lem}\label{s-dualD}
  Let $f\in\H^c_p(\real^d,\M)\cap L_2(\N)$ and $g\in \BMO_q^c(\real^d,\M)$. Then
 $$
 \Big|\tau \int _{\real^d} f(s)g^*(s)ds \Big|\les \|s^{c,D}_\Phi(f)\|_p^{\frac p 2}\,\|f\|_{\H^c_p}^{1-\frac p 2}\,\|g\|_{\BMO_q^c}\,.
 $$
 \end{lem}

\begin{proof}
 We use the truncated version of $s^{c,D}_\Phi(f)$:
 $$
 s^{c,D}_{\Phi}(f)(s,j) =\Big( \sum_{k=j}^{\8} |\Phi_{2^k}*f (s)|^2\Big)^{\frac12}\,.
 $$
The proof of Lemma~\ref{square1less2} is easily adapted to the present setting to ensure
 $$
 s^{c,D}_{\Phi}(f)(s,j)^2\les \sum_{m\in\nat_0^d, |m|_1\leq d} S^{c,D}_{D^m {\Phi}}(f)(s, j)^2\,.
 $$
Then
  $$\Big|\tau \int _{\real^d} f(s)g^*(s)ds \Big| ^2\le {\rm I'} \cdot {\rm II'},$$
where
 \be\begin{split}
 {\rm I'}&=\tau\int_{\real^d}\sum_j s_{\Phi}^{c,D}(f)(s,j) ^{p-2} |\Phi_{2^j}*f(s)|^2  ds \,,\\
 {\rm II'}&=\tau\int_{\real^d} \sum_j s_{\Phi}^{c,D}(f)(s,j)^{2-p} |\Psi_{2^j}*g(s)|^2 ds\,.
 \end{split}\ee
Both terms ${\rm I}' $ and ${\rm II}' $ are estimated exactly as before, so we have
  $${\rm I}' \le 2\|s_\Phi^c(f)\|_{p}^p\;\textrm{ and }\;
  {\rm II}'\les \|f\|_{\H^c_p}^{2-p}\,\|g\|_{\BMO_q^c}^2\,.$$
This gives the announced assertion.  \end{proof}

\begin{proof}[Proof of Theorem~\ref{equivalence HpD}]
 Armed with the preceding two lemmas and noting that Lemmas~\ref{SboundedH1toL1} and \ref{tent-complemented} also transfer to the discrete case with the same arguments, we prove Theorem~\ref{equivalence HpD} exactly in the same way as Theorem~\ref{equivalence Hp}.
 \end{proof}

\begin{rk}
 The proof of Theorem~\ref{equivalence HpD} also yields the discrete version of Theorem~\ref{CarlsonlessBMOqbis}.
 \end{rk}


\section{Proof of Theorem \ref{charact-Riesz of Poisson}}
\label{A special Phi}


We prove Theorem \ref{charact-Riesz of Poisson} in this section. First note that \eqref{D of Poisson} is a particular case of \eqref{Riesz of Poisson}. Indeed, by the inverse Fourier transform formula, we have
\be\begin{split}
 f* I^{k }(\mathrm{P})_{\e} (t)&= \int e^{{\rm i}2\pi t\cdot \xi}    \widehat{f}(\xi) |\e \xi|^{k }e^{-\e 2\pi |\xi|}  d\xi\\
&= \e^{k } \int e^{{\rm i}2\pi t\cdot \xi}  \widehat{f}(\xi) | \xi|^{k}e^{-\e 2\pi|\xi|}  d\xi\\
&= (-\frac{1}{ 2\pi})^k\e^{k } \frac{\partial^k}{\partial \e^k} \int e^{{\rm i}2\pi t\cdot \xi}  \widehat{f}(\xi)e^{-\e 2\pi|\xi|}  d\xi\\
&= (-\frac{1}{ 2\pi})^k\e^{k } \frac{\partial^k}{\partial \e^k} \big( \mathrm{P}_\e (f)(t)\big).
\end{split}\ee
Thus it remains to prove \eqref{Riesz of Poisson}. Before proceed further, let us note that \eqref{Riesz of Poisson} is the radial part of Theorem \ref{equivalence Hp} with $\Phi=I^{\a}(\mathrm{P})$.  The problem now is that this function $\Phi$ does not belong to the Schwartz class, so we cannot apply directly Theorem \ref{equivalence Hp}. However, we will show that the proof of that theorem works for this $\Phi$ too.

Reexamining the conditions of $\Phi$ that we have used in the proof of Theorem \ref{equivalence Hp}, we find that $\Phi \in \mathcal{S}$ is not necessary. Specifically, we collect all properties of $\Phi$ used there:
\begin{enumerate}[{\rm i)}]
\item Every $D^m \Phi$ with $0\leq |m|_1 \leq d$ makes $f\mapsto s_{D^m \Phi}^cf$ and $f\mapsto S_{D^m \Phi}^cf$  Calder\'on-Zygmund singular integral operators.
\item There exists a function $\Psi$ such that \eqref{reproduce} holds.
\item The above $\Psi$ makes $d\mu(f) =  |\Psi_\e*f (s)|^2\frac{d\e ds}{\e}$ a Carleson (or $q$-Carleson) measure, satisfying Theorem \ref{CarlsonlessBMOqbis}.
\end{enumerate}
Since $I^{\a }(\mathrm{P})$ is radial, one can always choose a radial Schwartz function $\Psi$ such that \eqref{reproduce} holds for $\Phi=I^{\a }(\mathrm{P})$.
Thus ii) and iii) above are fulfilled for $\Phi=I^{\a }(\mathrm{P})$.

It remains to show that $s_{D^m \Phi}^c$ and $S_{D^m \Phi}^c$ are Calder\'on-Zygmund singular integral operators.
To this end, we require a lemma.  For $\a\in\real$, let $J_\a$ be the function on $\real^d$ defined by $J_\a(s)=(1+|s|^2)^{\frac\a2}$, and let $J^\a$ be the Fourier multiplier of symbol $J_\a$. $J^\a$ is the Bessel potential of order $\a$ and $J^{\a}=(1-(2\pi)^{-2}\D)^{\frac{\a}{2}}$. The potential Sobolev space $H^\a_p(\real^d)$ of order $\a$ consists of distributions $f$ such that $J^\a(f)\in L_p(\real^d)$. It is clear that if $\a$ is a positive even integer, then the Sobolev space $W_p^\a(\real^d)$ is contained in $H^\a_p(\real^d)$.

\begin{lem}\label{Riesz of Poisson-CZ}
If $\a>0$, then $J_{d+\s}  I^{\a }(\mathrm{P})\in L_\8(\real^d) $ for some $\s>0$.
\end{lem}

\begin{proof}
First, consider the case $\a>1$. Then choose $\sigma=1$. We must show that $\big[J_{d+1}  I^{\a }(\mathrm{P})]^2$ is  a bounded function.  By the Hausdorff-Young inequality, it suffices to prove that
 $$J^{2(d+1)} \big[\wh{I^{\a }(\mathrm{P})}*\wh{I^{\a }(\mathrm{P})}\big] \in L_1(\real^d),\;\text{ or equivalently},\;
 \wh{I^{\a }(\mathrm{P})}*\wh{I^{\a }(\mathrm{P})}\in H^{2(d+1)}_1(\real^d).$$
Since  $W^{2(d+1)}_1(\real^d)\subset H^{2(d+1)}_1(\real^d)$, we are further reduced to showing  $\wh{I^{\a }(\mathrm{P})}*\wh{I^{\a }(\mathrm{P})}\in W^{2(d+1)}_1(\real^d)$. By easy calculations, for any $m=(m_1,..., m_d)\in\nat_0^d$, we have (recalling that $D^m$ is the partial derivation associated to $m$)
 $$D^m\big(|\xi|^\a e^{- 2\pi|\xi|}\big) \approx |\xi|^{\a-|m|_1 }\;\text{ for }\;\xi\in \real^d\;\text{ close to } 0.$$
It follows that $\wh{I^{\a }(\mathrm{P})}\in W_1^{d+1}(\real^d)$. Now any $m\in \nat_0^d$ with $|m|_1\le 2(d+1)$ can be decomposed into a sum $\el+n$ with $|\el|_1\le d+1$ and $|n|_1\le d+1$. Then
 $$D^m\big[\wh{I^{\a }(\mathrm{P})}* \wh{I^{\a }(\mathrm{P})}\big]
 =D^\el\big[\wh{I^{\a }(\mathrm{P})}\big]*D^n\big[\wh{I^{\a }(\mathrm{P})}\big].$$
Both partial derivatives on the right-hand side belong to $L_1(\real^d)$, so does the one on the left-hand side. We then deduce that $\wh{I^{\a }(\mathrm{P})}*\wh{I^{\a }(\mathrm{P})}\in W^{2(d+1)}_1(\real^d)$, as desired.

Next, note that the above reasoning also shows that $\big[J_{d}  I^{\a }(\mathrm{P})]^2\in L_\8(\real^d)$ for all $\a>0$, so $J_{d}  I^{\a }(\mathrm{P})\in L_\8(\real^d)$ .

Finally, we use the three lines lemma to handle the case $0<\a\le1$; then $\s$ can be any number in $(0,\,\a)$. We will need to allow $\a$ to take complex values in the preceding two parts which remain valid if ${\rm Re}(\a)>1$, respectively, if ${\rm Re}(\a)>0$. Now for any complex number $z$ in the strip $\{z\in\com: 0\le{\rm Re}(z)\le 1\}$ define
  $$F(z)=e^{(z-\sigma)^2}\, J_{d+z}\, I^{\a-\sigma+z}(\mathrm{P})\,.$$
Then
 $$\sup_{b\in\real}\big\|F({\rm i}b)\big\|_{L_\8(\real^d)}< \8 \;\text{ and }\;\sup_{b\in\real}\big\|F(1+{\rm i}b)\big\|_{L_\8(\real^d)}< \8.$$
It thus follows that
 $$J_{d+\s}\,I^{\a}(\mathrm{P})=F(\sigma)\in L_\8(\real^d).$$
This completes the proof of the lemma.
 \end{proof}

To finish the proof of  Theorem \ref{charact-Riesz of Poisson}, we are left to check that both square function operators  $s_{\Phi}^c$ and $S_{\Phi}^c$ are Calder\'on-Zygmund singular integral operators for $\Phi= I^{\a }(\mathrm{P})$. Take $s_{\Phi}^c$ as example. Since $\Phi\in L_\8(\real^d)$, we have
 $$|\Phi_\e(s)| =\frac1{\e^{d}}\big|\Phi(\frac{s}{\e})\big|\lesssim \frac1{\e^{d}}\,,\quad\forall s\in \real^d, \;|s|\le\e.$$
On the other hand,  Lemma \ref{Riesz of Poisson-CZ} ensures that
 $$|\Phi_\e(s)| =\frac1{\e^{d}}\big|\Phi(\frac{s}{\e})\big|\lesssim \frac{\e^\sigma}{|s|^{d+\sigma}} \,,\quad\forall s\in \real^d, \;|s|\ge\e.$$
The above two inequalities imply
 $$\Big( \int_0^\8 \big|\Phi_\e(s)\big|^2 \frac{d\e}{\e} \Big)^{\frac{1}{2}} \lesssim  \frac1{|s|^{d}}\,,\quad\forall s\in \real^d\setminus\{0\}.$$
Moreover,  a similar argument yields
 $$\Big( \int_0^\8 \big|\nabla[\Phi_\e(s)]\big|^2 \frac{d\e}{\e} \Big)^{\frac{1}{2}} \lesssim   \frac1{|s|^{d+1}}\,,\quad\forall s\in \real^d\setminus\{0\}.$$
Therefore,  the map $s\mapsto \Phi_\cdot(s)$ is a $H$-valued Calder\'on-Zygmund kernel, where $H=L_2((0,\,\8),\,\frac{d\e}{\e})$. Thus, $s_{\Phi}^c$ can be expressed as a singular integral operator.

In the same way, we  show that $s_{D^m\Phi}^c$ and $S_{D^m\Phi}^c$  are  Calder\'on-Zygmund operators too for all $m\in\nat_0^d$. Hence, \eqref{Riesz of Poisson} is proved.


\section{Applications to tori}
\label{Applications to tori}


Mei's work \cite{Mei2007} has been extended to the torus case in  \cite{CXY2012} with a view to applications to the quantum setting. Note, however, that this extension is not straightforward. The main idea is to reduce the torus case to the Euclidean one in order to use Mei's arguments. We now recall the relevant definitions and results. Let $\T^d$ denote  the $d$-torus with normalized Haar measure $dz$, and let $\N=L_\8(\T^d)\overline\ot\M$ throughout this section.

A cube of $\T^d$ is a product $Q=I_1\times\cdots\times I_d$, where each $I_j$ is an interval (= arc) of $\T$. As in the Euclidean case, we use $|Q|$ to denote the normalized volume (= measure) of $Q$. The whole $\T^d$ is now a cube too (of volume $1$). We then define $\BMO^c(\T^d, \M)$ as the space of all $f\in L_2(\N)$ such that
 $$\|f\|_{\BMO^c}=\max\big\{\big\|f_{\T^d}\big\|_\M,\;
 \sup_{Q\subset \T^d  \text{cube}} \Big\|\frac1{|Q|}\int_Q\big|f(z)-f_Q\big|^2dz\Big\|_\M^{\frac12}\big\}<\8.$$
This  is a Banach space. The row and mixture spaces $\BMO^r(\T^d, \M)$  and $\BMO(\T^d, \M)$ are defined by taking adjoints and intersection like in the Euclidean case.

A very useful property of $\BMO^c(\T^d, \M)$ is its embedding  into $\BMO^c(\real^d, \M)$ via periodization. To state this property, we will identify $\T^d$ with the unit cube $\mathbb{I}^d=[0,\, 1)^d$ via $(e^{2 \pi \mathrm{i} s_1},\cdots , e^{2 \pi \mathrm{i} s_d})\leftrightarrow(s_1, \cdots, s_d)$. Under this identification, the addition in $\mathbb{I}^d$ is the usual addition modulo $1$ coordinatewise; an interval of $\mathbb{I}$ is either a subinterval of $\mathbb{I}$ or a union  $[b,\, 1]\cup[0,\, a]$ with $0<a<b<1$, the latter union being the interval $[b-1,\, a]$ of $\I$ (modulo $1$). So the cubes of $\mathbb{I}^d$ are exactly those of $\T^d$.  Accordingly, functions on $\T^d$ and $\mathbb{I}^d$ are identified too. Thus  $\N=L_\8(\T^d)\overline\ot\M= L_\8(\I^d)\overline\ot\M$.

Functions on $\T^d$ are $1$-periodic functions on $\real^d$, or equivalently, functions on $\I^d$ can be extended to  $1$-periodic functions on $\real^d$. We will identify functions on $\T^d$  or  $\I^d$ as  $1$-periodic functions on $\real^d$. However, for clarity and if necessary, we will write $f_{\rm pe}$ when $f$ is considered as a $1$-periodic function on $\real^d$ for $f$ on $\T^d$. It is proved in \cite{CXY2012} that  modulo constant functions, $\BMO^c(\T^d, \M)$ embeds into $\BMO^c(\real^d, \M)$ via the map $f\mapsto f_{\rm pe}$.  More precisely, for any $f\in L_2(\N)$ we have
 \beq\label{BMO periodic}\begin{split}
 \sup_{Q\subset \T^d  \text{ cube}} \Big\|\frac1{|Q|}\int_Q\big|f(z)-f_Q\big|^2dz\Big\|_\M^{\frac12}
 &=\sup_{Q\subset \I^d  \text{ cube}} \Big\|\frac1{|Q|}\int_Q\big|f(s)-f_Q\big|^2ds\Big\|_\M^{\frac12}\\
 &\approx
 \|f_{\rm pe}\|_{\BMO^c(\real^d, \M)}
 \end{split}\eeq
with relevant constants depending only on $d$. This property enables us to reduce the treatment of  $\BMO^c(\T^d, \M)$ to the Euclidean setting.

In order to give an intrinsic definition of BMO in the quantum case, we will need another characterization of $\BMO^c(\T^d, \M)$ by the circular Poisson semigroup. Let $\mathbb{P}_r$ denote the circular Poisson kernel of  $\T^d$:
 \beq\label{circular P}
 \mathbb{P}_r(z) = \sum_{m \in \ent^d }r^{|m|} z^{m}, \quad z\in\T^d,\; 0 \le r < 1.
 \eeq
Then for any $f\in L_1(\N)$, its Poisson integral is
 $$\mathbb{P}_r(f) (z)=\int_{\T^d}\mathbb{P}_r(zw^{-1})f(w)dw=
 \sum_{m \in \mathbb{Z}^d } \wh f(m) r^{|m|} z^{m}.$$
Here $\wh f$ denotes, of course, the Fourier transform of $f$:
 $$\wh f(m)=\int_{\T^d}f(z)\, z^{-m}dz.$$
It is proved in \cite{CXY2012} that
 \beq\label{Poisson-BMO}
 \sup_{Q\subset \T^d  \text{ cube}} \Big\|\frac1{|Q|}\int_Q\big|f(z)-f_Q\big|^2dz\Big\|_\M\approx
 \sup_{0\le r<1}\big\|\mathbb{P}_r(|f-\mathbb{P}_r(f)|^2)\big\|_{\N}
 \eeq
with relevant constants depending only on $d$. Thus
 $$
 \|f\|_{\mathrm{BMO}^c}\approx \max\big\{\|\wh f(0)\|_\8,\;
 \sup_{0\le r<1}\big\|\mathbb{P}_r(|f-\mathbb{P}_r(f)|^2)\big\|_{\N}^{\frac12}\big\}.
 $$

Now we turn to  the operator-valued Hardy spaces on $\T^d$ which are defined by the Littlewood-Paley or Lusin square functions associated to the circular Poisson kernel. We will use the same notation $s^c$ and $S^c$ to denote these square functions. This should not cause any confusion in concrete contexts.  For $f\in L_1(\N)+L_\8(\N)$ define
 \beq\label{LP periodic}
 s^c(f) (z)= \Big( \int_0^1 \big |\frac{\partial}{\partial r}\,\mathbb{P}_r(f)(z)\big |^2(1-r)dr\Big)^{\frac12}\,,\quad z\in\T^d.
 \eeq
This is the torus analogue of the radial square function defined by \eqref{LP}.
For $1\leq p <\infty$, let
 $$\H^c_p(\T^d, \M)=\{f\in L_1(\N)+L_\8(\N): \|f\|_{\H^c_p}<\8\},$$
where
 $$\|f\|_{\H^c_p}=\|\wh f(0)\|_{L_p(\M)} +\|s^c(f)\|_{L_p(\N)}.$$
The row Hardy space  $\H^r_p(\T^d, \M)$ is defined to be the space of all $f$ such that $f^*\in\H^c_p(\T^d, \M)$, equipped with the natural norm. Then we define
  \be
  \H_p(\T^d, \M) =
 \left \{ \begin{split}
 &\H_p^c(\T^d, \M)+ \H_p^r(\T^d, \M)& \textrm{ if }\; 1\le p<2,\\
 &\H_p^c(\T^d, \M) \cap \H_p^r(\T^d, \M)  & \textrm{ if }\; 2\le  p<\8,
 \end{split} \right.
 \ee
equipped with the sum and intersection norms, respectively.

Like in the Euclidean case, the Littlewood-Paley $g$-function above can be replaced by the Lusin area integral function. For  $z\in\T^d$ let $\D(z)$ be the Stoltz domain with vertex $z$ and aperture $2$:
 $$\D(z)=\{w\in\com^d\;:\; |z-w|\le 2(1-|w|)\}.$$
 For $f\in L_1(\N)+L_\8(\N)$ define the torus counterpart of \eqref{Lusin} by
 \beq\label{Lusin periodic}
 S^c(f) (z)= \Big(\int_{\D(z)} \big |\frac{\partial}{\partial r}\,\mathbb{P}_r(f)(rw)\big |^2\frac{dwdr}{(1-r)^{d-1}}\Big)^{\frac12}\,,
 \quad z\in\T^d,
 \eeq
where the integral is taken  on $\D(z)$ with respect to $rw\in\D(z)$ with $0\le r<1$ and $w\in\T^d$.

Like for BMO spaces, we use periodization to deal with Hardy spaces on $\T^d$ too. Following the discussion and convention before \eqref{BMO periodic}, considered as a $1$-periodic function on $\real^d$,  the Poisson integral $\mathbb P_r(f)$ of $f$ on $\T^d$ coincides with the Poisson integral $\mathrm P_\e(f)$ on $\real^d$ (the latter $f$ being viewed as a $1$-periodic function on $\real^d$). More precisely,
 $$\mathbb P_r(f)(z)=\mathrm P_\e(f_{\rm pe})(s) \; \textrm{ with }\; z=(e^{2\pi{\rm i}s_1},\cdots, e^{2\pi{\rm i}s_d})\;\textrm{ and }\; r=e^{-2\pi\e}\,.$$
This is an immediate consequence of the classical
 Poisson summation formula (see \cite[Corollary~VII.2.6]{SW1975}):
 \beq\label{Poisson summation}
 \mathbb P_r(z)=\sum_{m\in\ent^d}\mathrm P_\e(s+m)\; \textrm{ with }\; z=(e^{2\pi{\rm i}s_1},\cdots, e^{2\pi{\rm i}s_d})
 \;\textrm{ and }\; r=e^{-2\pi\e}\,.
  \eeq
In what follows, we will always assume that $z$ and $s$, $r$ and $\e$ are related as above.

The preceding periodization property of the Poisson integrals can be reformulated on $\mathbb{I}^d$.  Let $\wt{\mathrm{P}}_\e(s)$ denote the right-hand side of \eqref{Poisson summation}, that is,   $\wt{\mathrm{P}}_\e$ is the $1$-periodization of $\mathrm{P}_\e$.
With the identification between functions on $\T^d$ and $\I^d$,  we have $\wt{\mathrm{P}}_\e=\mathbb P_r$ with $ r=e^{-2\pi\e}$. Thus
 $$\mathbb P_r(f)(z)= \wt{\mathrm{P}}_\e(f)(s)= \wt{\mathrm{P}}_\e*f(s)=\int_{\mathbb{I}^d } \wt{\mathrm{P}}_\e(s-t)f(t)dt.$$
It then follows that
 \beq\label{s per}
 s^c(f_{\rm pe}) (s)= \Big(\int_0^\8 \big |\frac{\partial}{\partial \e}\, \wt{\mathrm{P}}_\e(f)(s) \big |^2\e d\e\Big)^{\frac12} \,,\quad s\in\I^d.
 \eeq
Note that here $s^c(f_{\rm pe}) $ is the radial square function on $\real^d$ defined by \eqref{LP} since $f_{\rm pe}$ is a function on $\real^d$.  Similarly,
 \beq\label{S per}
 S^c(f_{\rm pe}) (s)=
  \Big ( \int_{\Ga} \big |\frac{\partial}{\partial\e}\, \wt{\mathrm{P}}_\e(f)(s+t) \big |^2 \frac{dtd\e}{\e^{d-1}} \Big )^{\frac12}
  \,,\quad s\in\I^d.
  \eeq
The two equalities above, together with \eqref{Poisson summation}, establish the link between the square functions on $\T^d$ and $\real^d$.  More precisely, for $1\le p<\8$ and any function $f$ on $\T^d$ we have
 \beq\label{square periodic}
 \|s^c(f)\|_{L_p(\N)}\approx \|s^c(f_{\rm pe})\|_{L_p(\I^d; L_p(\M))}
 \;\text{ and }\;
 \|S^c(f)\|_{L_p(\N)}\approx \|S^c(f_{\rm pe})\|_{L_p(\I^d; L_p(\M))}\,.
 \eeq
Recall that  $L_p(\N)=L_p(\T^d; L_p(\M))$. So with the identification $\T^d\approx\I^d$,  the norms above in both sides coincide. Here, we have written explicitly $\|\,\|_{L_p(\I^d; L_p(\M))}$ in order to emphasize the fact that although $s^c(f_{\rm pe})$ and $S^c(f_{\rm pe})$ are defined on $\real^d$, the two norms on the right-hand sides are restricted only to $\I^d$.

The equivalence relations \eqref{BMO periodic} and \eqref{square periodic} allow us to reduce the treatment of $\T^d$ to that of $\real^d$, so to follow the arguments of \cite{Mei2007}. A major difference compared with \cite{Mei2007} is that all considerations are now restricted to the cube $\I^d$ instead of the whole $\real^d$. In this way, we proved in \cite{CXY2012} the following result  which is the torus counterpart of the main result of  \cite{Mei2007}.

\begin{thm}\label{H1-BMO bis}
  \begin{enumerate}[\rm (i)]
 \item The dual space of $\H_1^c(\T^d, \M)$  coincides isomorphically with $\BMO^c(\T^d, \M)$.
  \item Let $1\le p<\8$. Then for any $f\in L_1(\N)+L_\8(\N)$
 $$\big\|s^c(f) \big\|_{p}\approx \big\|S^c(f) \big\|_{p}$$
with relevant constants depending only on $d$ and $p$.
 \item  Let $1<p<\8$. Then $\H_p(\T^d, \M)=L_p(\N)$ with equivalent norms.
 \item  Let $1<p<\8$. Then
 $$(\BMO^c(\T^d, \M),\; \H^c_1(\T^d, \M))_{\frac1p}
 =\H^c_p(\T^d, \M)\;\textrm{ with equivalent norms}.$$
 \end{enumerate}
  \end{thm}

Like in the previous sections we wish to characterize the Hardy spaces on $\T^d$ by square functions defined by any Schwartz function instead of the Poisson kernel. Let $\Phi$ be a Schwartz function of vanishing mean and satisfying the following condition:
 \beq\label{schwartz per}
 \forall\, \xi\in\real^d\;\text{ with }\; |\xi|\ge1\;\; \exists \, \e\in (0, \,1) \;\text{ s.t. } \; \wh\Phi(\e\xi)\neq0.
 \eeq
Then there exists another Schwartz function $\Psi$ of vanishing mean such that
 $$\int_{0}^1\wh\Phi(\e\xi)\, \overline{\wh\Psi(\e\xi)}\,\frac{d\e}\e=1,\quad \xi\in\real^d\;\text{ with }\; |\xi|\ge1.$$
 Let $\wt\Phi_\e$ be the periodization of $\Phi_\e$:
 $$\wt\Phi_\e(s)=\sum_{m\in\ent^d}\Phi_\e(s+m).$$
Then for  $f\in L_1(\N)+L_\8(\N)$,
  $$
 \wt\Phi_\e (f)(s)=\int_{\mathbb{I}^d }\wt\Phi_\e(s-t)f(t)dt
 =\sum_{m \in \mathbb{Z}^d }  \wh{\Phi}(\e m)\wh{f} (m) z^{m}\,,\;\;z=(e^{2\pi{\rm i}s_1},\cdots, e^{2\pi{\rm i}s_d}).
$$
The radial and conic square functions of $f$ associated to $\Phi$ are defined by
 $$s_{\Phi}^c(f) (s)^2= \int_0^\8 \big |\wt\Phi_\e(f)(s)\big |^2\,\frac{d\e}{\e}\,,\quad
 S_{\Phi}^c(f)(s)^2= \int_{\Ga} \big |\wt\Phi_\e(f)(s+t)\big |^2\frac{dtd\e}{\e^{d+1}}\,,\quad s\in\I^d.$$
In the present case of $\T^d$, the first integral above can now be restricted to  the unit interval $(0,\, 1)$ without changing the norm of $s_{\Phi}^c(f)$ in $L_p(\N)$. More precisely, we have the following:

\begin{prop}
 Let
   $$ \wt{s^c}(f) (s)^2= \int_0^1 \big |\wt\Phi_\e(f)(s)\big |^2\,\frac{d\e}{\e}\,,\quad
 \wt{S_{\Phi}^c}(f)(s)^2= \int_{\wt{\Ga}} \big |\wt\Phi_\e(f)(s+t)\big |^2\frac{dtd\e}{\e^{d+1}}\,,$$
 where $\wt{\Ga}$ is the truncated cone: $\wt{\Ga}=\Ga\cap\big( \real^d\times (0,\, 1)\big)$. Then for any $1\le p<\8$,
  $$\|s^c(f)\|_{L_p(\N)}\approx \|\wt{s^c}(f)\|_{L_p(\N)}\,,\quad \|S^c(f)\|_{L_p(\N)}\approx \|\wt{S^c}(f)\|_{L_p(\N)}\,,$$
   where the equivalence constants depend only on $d$ and $\Phi$.
   \end{prop}

 \begin{proof}
  The proof is elementary. We consider only the radial square function, the conic one being treated similarly. Fix an $f\in L_1(\N)$ with $\wh f(0)=0$. For any $m\in\ent^d\setminus\{0\}$, we have
   \be\begin{split}
   \|\wt{s^c}(f)\|_{L_p(\N)}
   &\ge\|\wh f(m)\|_{L_p(\M)}\,\Big(\int_0^1 \big|\wh\Phi(\e m)\big|^2\,\frac{d\e}\e\Big)^{\frac12}\\
   &\ge \|\wh f(m)\|_{L_p(\M)}\,\Big(\int_0^{|m|} \big|\wh\Phi(\e \frac{m}{|m|})\big|^2\,\frac{d\e}\e\Big)^{\frac12}\\
   &\ge a\,\|\wh f(m)\|_{L_p(\M)}\,,
   \end{split}\ee
  where
  $$a=\inf_{\xi\in\real^d,\, |\xi|=1}\, \Big(\int_0^{1}\big|\wh\Phi(\e \xi)\big|^2\frac{d\e}\e\Big)^{\frac12}\,.$$
 The assumption \eqref{schwartz per} ensures that the above integral is positive for every $\xi$ in the unit sphere of $\real^d$; so by continuity and compactness, $a>0$. Thus
   $$\sup_{m\in\ent^d\setminus\{0\}}\|\wh f(m)\|_{L_p(\M)}\le \frac1a\,\|\wt{s^c}(f)\|_{L_p(\N)}\,.$$
  On the other hand, for sufficiently large $\s$,
   \be\begin{split}
   \Big\|\Big(\int_1^\8 \big |\wt\Phi_\e(f)(s)\big |^2\,\frac{d\e}{\e}\Big)^{\frac12}\Big\|_{L_p(\N)}
   &\le \sum_{m\in\ent^d\setminus\{0\}} \|\wh f(m)\|_{L_p(\M)}\,\Big(\int_1^\8 \big|\wh\Phi(\e m)\big|^2\,\frac{d\e}\e\Big)^{\frac12}\\
   &\les\sum_{m\in\ent^d\setminus\{0\}} \|\wh f(m)\|_{L_p(\M)}\,\Big(\int_1^\8 \frac1{|\e m|^{2\s}}\,\frac{d\e}\e\Big)^{\frac12}\\
   &\les\sup_{m\in\ent^d\setminus\{0\}}\|\wh f(m)\|_{L_p(\M)} \,\sum_{m\in\ent^d\setminus\{0\}} \frac1{|m|^\s}\\
   &\les\sup_{m\in\ent^d\setminus\{0\}} \|\wh f(m)\|_{L_p(\M)}\,,
   \end{split}\ee
  We then deduce the desired assertion.
   \end{proof}

Combining the preceding periodization argument and those of section~\ref{Proof1}, we obtain the following characterization of $\H^c_p(\T^d,\M)$ by $s^c_\Phi$ and $S^c_\Phi$, which is the main result of this section.

 \begin{thm}\label{HpPhi}
Let  $1\le p<\8$ and $f\in L_1(\N)+L_\8(\N)$. Then
$$\|f\|_{\mathcal{H}_p^c}\approx \|\wh f(0)\|_{L_p(\M)}+ \| s^c_\Phi(f)\|_{L_p(\N)}
\approx \|\wh f(0)\|_{L_p(\M)}+ \| S_{\Phi}^c(f)\|_{L_p(\N)}\,.$$
\end{thm}

We can also prove the discrete version of the above theorem. Define the following discrete analogue of $s^c_\Phi$ on $\T^d$ (leaving that of $S^c_\Phi$ to the reader):
 $$s^{c, D}_\Phi(f)(s)=\Big(\sum_{j\ge0}\big|\wt \Phi_{2^{-j}}(f) (s)\big|^2\Big)^{\frac12}\,,\quad s\in\I^d.$$
Like for Theorem \ref{equivalence HpD}, we now need to reinforce the assumption on $\Phi$ that is now supposed to satisfy
 \beq\label{schwartz per D}
 \wh\Phi\neq0 \;\text{ on }\; \{\xi\in\real^d\,:\, 1\le|\xi|<2\}.
 \eeq
Then there exists a Schwartz function $\Psi$ of vanishing mean such that
 $$\sum_{j\ge0}\wh\Phi(2^{-j}\xi)\,\overline{\wh\Psi(2^{-j}\xi)}=1,\quad \xi\in\real^d\;\text{ with }\; |\xi|\ge1.$$
Using the arguments of section~\ref{Proof2} and periodization, we obtain

\begin{thm}\label{HpPhiD}
Let  $1\le p<\8$ and $f\in L_1(\N)+L_\8(\N)$. Then
$$\|f\|_{\mathcal{H}_p^c}\approx \|\wh f(0)\|_{L_p(\M)}+ \| s^{c, D}_\Phi(f)\|_{L_p(\N)}\,.$$
\end{thm}

Like for Theorem \ref{charact-Riesz of Poisson}, the function $\Phi$  can be taken to be $I^{\a}(\mathrm{P})$ with $\a>0$:

\begin{thm}\label{HpPoisson}
 Both Theorems~\ref{HpPhi} and \ref{HpPhiD} hold for $\Phi=I^{\a}(\mathrm{P})$ with $\a>0$.
 \end{thm}


\section{Applications to quantum tori}
\label{Applications to quantum tori}


We now apply the results of the previous section to the quantum case. To this end, we first recall the relevant definitions.
Let $d\ge2$ and $\theta=(\theta_{kj})$ be a real skew symmetric $d\times d$-matrix. The associated $d$-dimensional noncommutative
torus $\mathcal{A}_{\theta}$ is the universal C*-algebra generated by $d$
unitary operators $U_1, \ldots, U_d$ satisfying the following commutation
relation
 $$U_k U_j = e^{2 \pi \mathrm{i} \theta_{k j}} U_j U_k,\quad j,k=1,\ldots, d.$$
We will use standard notation from multiple Fourier series. Let  $U=(U_1,\cdots, U_d)$. For $m=(m_1,\cdots,m_d)\in\ent^d$ define
 $$U^m=U_1^{m_1}\cdots U_d^{m_d}.$$
 A polynomial in $U$ is a finite sum
  $$ x =\sum_{m \in \mathbb{Z}^d}\alpha_{m} U^{m}\quad \text{with}\quad
 \alpha_{m} \in \mathbb{C}.$$
The involution algebra
$\mathcal{P}_{\theta}$ of all such polynomials is
dense in $\A_{\theta}.$ The functional $x\mapsto\a_0$ on $\mathcal{P}_{\theta}$ extends to a  faithful  tracial state $\tau$ on $\A_{\theta}$.  Let  $\mathbb{T}^d_{\theta}$ be the w*-closure of $\A_{\theta}$ in  the GNS representation of $\tau$. This is our $d$-dimensional quantum torus. The state $\tau$ extends to a normal faithful tracial state on $\mathbb{T}^d_{\theta}$ that will be denoted again by $\tau$.  Note that if $\theta=0$, then  $\mathbb{T}^d_{\theta}=L_\8(\T^d)$ and $\tau$ coincides with the integral on $\T^d$ against normalized Haar measure $dz$.

Any $x\in L_1(\T^d_\theta)$ admits a formal
Fourier series:
 $$x \sim \sum_{m \in\mathbb{Z}^d} \widehat{x} ( m ) U^{m}\;\text{ with }\; \widehat{x}( m) = \tau((U^m)^*x).$$

We introduced in  \cite{CXY2012}  a transference method to overcome the full noncommutativity of quantum tori and use methods of operator-valued harmonic analysis.  Let $\mathcal{N}_{\theta} = L_{\infty} (\mathbb{T}^d) \overline{\otimes}\mathbb{T}^d_{\theta}$, equipped with the tensor trace $\nu = \int dz \otimes \tau$. For each $z
\in \mathbb{T}^d,$ define $\pi_{z}$ to be the isomorphism of
$\mathbb{T}^d_{\theta}$ determined by
 $$\pi_{z} (U^{m}) = z^{m} U^{m } = z_1^{m_1} \cdots z_d^{m_d}U_1^{m_1} \cdots U_d^{m_d}. $$
This isomorphism preserves the
trace $\tau.$ Thus for every $1\le p < \8$,
 $$\|\pi_{z} (x) \|_p = \|x\|_p,\; \forall x\in L_p(\mathbb{T}^d_{\theta}).$$

The main points of the transference method are contained in the following lemma from   \cite{CXY2012}.

\begin{lem}\label{prop:TransLp}
 \begin{enumerate}[{\rm i)}]
\item For any $x \in L_p (\mathbb{T}^d_{\theta})$, the function
$\tilde{x}: z \mapsto \pi_z(x)$ is continuous from
$\mathbb{T}^d$ to $L_p(\mathbb{T}^d_{\theta})$ $($with respect to the w*-topology for $p=\8)$.
\item Let $1\le  p \leq \8.$ If $x\in L_p (\mathbb{T}^d_{\theta}),$ then $\tilde{x} \in L_p (\N_{\theta})$ and $\| \tilde{x} \|_p = \| x\|_p,$ that is, $x \mapsto \tilde{x}$ is an isometric embedding from $L_p (\mathbb{T}^d_{\theta})$ into $L_p (\N_{\theta})$.

\item Let $\widetilde{\mathbb{T}^d_{\theta}} = \{ \tilde{x}: x \in \mathbb{T}^d_{\theta}\}.$ Then $\widetilde{\mathbb{T}^d_{\theta}}$ is a von Neumann subalgebra of $\mathcal{N}_{\theta}$ and the associated conditional expectation is given by
 $$\mathbb{E} (f)(z) = \pi_{z} \Big ( \int_{\mathbb{T}^d} \pi_{\overline{w}}
 \big [ f( w )\big ] dw \Big ),\quad z\in\T^d, \; f \in \N_{\theta}.$$
Moreover, $\mathbb{E}$ extends to a contractive projection from $L_p(\N_{\theta})$ onto $L_p(\widetilde{\mathbb{T}^d_{\theta}})$ for $1\leq p\leq \infty$.
\end{enumerate}
\end{lem}

The transference method consists in the following procedure:
\be
x\in L_p(\mathbb{T}^d_{\theta}) \mapsto \tilde{x}\in
L_p(\widetilde{\mathbb{T}^d_{\theta}}) \subset
L_p(\mathcal{N}_{\theta}).\ee
This allows us to work in
$L_p(\mathcal{N}_{\theta}).$ Then using the conditional expectation
$\mathbb{E}$ to return back to
$L_p(\widetilde{\mathbb{T}^d_{\theta}}) \cong
L_p(\mathbb{T}^d_{\theta}).$

We will use the same symbol $\mathbb{P}_r$ to denote the circular Poisson kernel on the quantum torus $\T^d_\theta$ too. Thus for any $x\in L_1(\T^d_{\theta})$
 $$\mathbb{P}_r(x) = \sum_{m \in \mathbb{Z}^d } \wh{x} ( m ) r^{|m|} U^{m}, \quad 0 \le r < 1.$$
The associated Littlewood-Paley $g$-function is
 $$s^c(x) =  \Big(\int_0^1 \big|\frac{\partial}{\partial r}\,\mathbb{P}_r(x) \big |^2(1-r)dr \Big )^{\frac 1 2}.$$
We leave to the reader to formulate the analogue of the Lusin square function. For $1\leq p <\infty$ let
 $$\|x\|_{\H^c_p}=|\wh x(0)| +\|s^c(x)\|_{L_p(\mathbb{T}^d_{\theta})}.$$
The column Hardy space $\mathcal{H}^c_p(\mathbb{T}^d_{\theta})$ is defined to be
 $$\H^c_p(\mathbb{T}^d_{\theta})=\big\{x\in L_1(\T^d_{\theta}): \|x\|_{\H^c_p}<\8 \big\}.$$
On the other hand, inspired by \eqref{Poisson-BMO}, we define
 $$
 \BMO^c(\mathbb{T}_{\theta}^d) =  \big\{x\in L_2(\mathbb{T}_{\theta}^d)\;:\; \sup_{0\le r<1}
  \big\|\mathbb{P}_{r} \big(|x-\mathbb{P}_r(x)|^2\big) \big\|_{\mathbb{T}^d_{\theta}}<\infty \big\},
 $$
equipped with the norm
 $$\|x\|_{\mathrm{BMO}^c}=\max\big\{|\wh x(0)|,\;
 \sup_{0\le r<1}  \big\|\mathbb{P}_r \big(|x-\mathbb{P}_r(x)|^2 \big) \big\|_{\mathbb{T}^d_{\theta}}^{\frac12}\,\big\}.$$
The corresponding row and mixture spaces are defined similarly.

Using transference, we can easily show that the map $x\mapsto\tilde x$ in Lemma~\ref{prop:TransLp} extends to an isometric embedding from $\H^c_p(\T^d_{\theta})$ into  $\H^c_p(\T^d, \T^d_{\theta})$ for any $1\le p<\8$ and from  $\BMO^c(\mathbb{T}_{\theta}^d)$ into  $\BMO^c(\T^d, \T^d_{\theta})$. Moreover, the ranges of these embeddings are $1$-complemented in their respective spaces (see \cite{CXY2012}).
This transference result immediately implies that Theorem~\ref{H1-BMO bis} remains valid in the quantum setting. In particular,  $\H_p(\mathbb{T}^d_\theta)=L_p(\mathbb{T}^d_\theta)$ with equivalent norms for $1<p<\8$.

The same argument allows us to show that the circular Poisson kernel can be replaced by a Schwartz function $\Phi$ of vanishing mean satisfying \eqref{schwartz per}. Like in the previous section, for $x\in L_1(\T^d_{\theta})$ define
 $$\wt\Phi_{\e}(x) =\sum_{m \in \mathbb{Z}^d }  \wh{\Phi}(\e m)\wh{x} ( m ) U^{m}$$
and
 $$s_{\Phi}^c(x)^2= \int_0^1 \big |\wt\Phi_{\e}(x)\big |^2\,\frac{d\e}{\e}\,,\quad
 s^{c, D}_\Phi(x)^2=\sum_{j\ge1}\wt\Phi_{2^{-j}}(x).$$
Together with the transference, Theorems~\ref{HpPhi}, \ref{HpPhiD} and \ref{HpPoisson}  imply the following

 \begin{thm}\label{q-Phi equivalence}
 Let $1\le p<\8$.
 \begin{enumerate}[\rm i)]
 \item Assume that $\Phi\in\mathcal S$  is of zero mean and satisfies \eqref{schwartz per}. Then for any  $x\in L_1(\T^d_{\theta})$,
 $$\|x\|_{\mathcal{H}_p^c}\approx |\wh x(0)|+ \| s^{c}_\Phi(x)\|_{L_p(\T^d_{\theta})}$$
 with relevant constants depending only on  $d, p$ and  $\Phi$.
\item If additionally $\Phi$   satisfies \eqref{schwartz per D}, then $s^{c}_\Phi$ can be replaced by  $s^{c, D}_\Phi$ in the above assertion.
\item Both assertions {\rm i)} and {\rm ii)} continue to hold for $\Phi=I^{\a}(\mathrm{P})$ with $\a>0$.
\end{enumerate}
\end{thm}

\n{\bf Acknowledgements.} We are very grateful to Tao Mei for many useful discussions.  We acknowledge the financial supports of ANR-2011-BS01-008-01, NSFC grant (No. 11271292 and 11301401).


\end{document}